\documentclass[final,3p,times]{elsarticle}

\usepackage{algorithm}
\usepackage{algorithmic}
\usepackage{amsmath}
\usepackage{amsfonts}
\usepackage{amsmath}
\usepackage{amssymb}
\usepackage{amsthm}
\usepackage{array}
\usepackage{bm}
\usepackage{booktabs}
\usepackage[autostyle=false, style=english]{csquotes}
\usepackage{enumitem}
\usepackage{multirow}
\usepackage[short,nocomma]{optidef}     
\usepackage{orcidlink}
\usepackage{siunitx}
\usepackage{subcaption}
\usepackage{tikz}
\usetikzlibrary{positioning, calc, decorations.pathreplacing}
\usepackage{todonotes}

\usepackage[capitalise]{cleveref}

\crefname{equation}{}{}
\Crefname{equation}{Equation}{Equations}
\crefname{table}{Tab.}{Tab.}
\Crefname{table}{Tab.}{Tab.}
\crefname{section}{Sec.}{Sec.}
\Crefname{section}{Sec.}{Sec.}

\newcommand{\cI}{{\mathcal{I}}}

\newcommand{\cT}{{\mathcal{T}}}

\newcommand{\R}{\mathbb{R}}
\newcommand{\N}{\mathbb{N}}

\newcommand{\A}{\mathcal{A}}
\newcommand{\F}{\mathcal{F}}

\newcommand{\rev}[1]{{\color{black} #1}} 

\newtheorem{corollary}{Corollary}
\newtheorem{theorem}{Theorem}
\newtheorem{prbl}{Problem}
\newtheorem{lemma}{Lemma}
\newtheorem{proposition}{Proposition}
\newtheorem{assumption}{Assumption}
\newtheorem{definition}{\textbf{Definition}}

\newlist{Aenumerate}{enumerate}{1}
\setlist[Aenumerate]{label=\arabic*)}

\journal{Nonlinear Analysis: Hybrid Systems}

\begin{document}

\begin{frontmatter}

\title{Continuous-Time Constrained Linear Quadratic Regulator \\
for Switched Linear Systems}
\tnotetext[t1]{This research is partially supported by the Italian Ministry of University and Research (MUR) - Fondo Italiano per la Scienza Applicata (FISA), through the OCCAM project under Grant FISA 2023-00324 CUP I53C25000700001 and the “FoReLab" (Future-oriented Research Lab) Project (Departments of Excellence), and partially by the European Union HORIZON-MSCA-2023-SE-01-01 - MSCA Staff Exchanges 2023 Program under the Grant Agreement No.101182891 (NEUTRAWEED). \\
This is the author accepted manuscript of an article published in \textit{Nonlinear Analysis: Hybrid Systems}.
The final published version is available at:
\url{https://doi.org/10.1016/j.nahs.2025.101625}
This manuscript version is made available under the
CC BY-NC-ND 4.0 license:
\url{https://creativecommons.org/licenses/by-nc-nd/4.0/}.
}

\author[1,2]{Pietro~Gori~\orcidlink{0009-0002-4989-321X}\corref{cor1}}
\ead{pietro.gori@phd.unipi.it}
\author[1,2]{Michele~Pierallini~\orcidlink{0000-0003-0547-2747}}
\ead{michele.pierallini@phd.unipi.it}
\author[1,2]{Franco~Angelini~\orcidlink{0000-0003-2559-9569}}
\ead{frncangelini@gmail.com}
\author[1,2]{and~Manolo~Garabini~\orcidlink{0000-0002-5873-3173}}
\ead{manolo.garabini@gmail.com}
\cortext[cor1]{Corresponding author}

\affiliation[1]{organization={Centro di Ricerca ``Enrico Piaggio'', Universita' di Pisa},
            addressline={Largo Lucio Lazzarino 1}, 
            city={Pisa},
            postcode={56126}, 
            country={Italy}}
\affiliation[2]{organization={Dipartimento di Ingegneria dell'Informazione,  Universita' di Pisa},
            addressline={Via G. Caruso 56}, 
            city={Pisa},
            postcode={56122},
            country={Italy}}

\begin{abstract}
\rev{Switched systems, characterized by a family of subsystems governed by a switching rule, widely apply to complex real-world scenarios. 
However, their inherent switching dynamics pose significant challenges in phase analysis and control design. 
To address these challenges, we propose a framework for solving the Continuous-Time Constrained Linear Quadratic Regulator (CT-CLQR) problem for switched linear systems.
Our approach partitions the time horizon into a finite number of intervals, each associated with a specific system mode. 
The duration of these intervals is parameterized by the switching instants, enabling a reformulation of the problem. 
We indirectly optimize the switching sequence by fixing the switching sequence and optimizing the interval durations.
We derive analytical expressions for the cost function and its gradient, which are critical for efficient optimization. 
Unlike state-of-the-art methods that impose equality constraints on state evolution, our approach inherently considers the state evolution in the cost function. 
This not only simplifies the problem formulation but also reduces computational overhead by precomputing shared terms offline, enhancing efficiency during online operations.
The proposed method significantly advances existing techniques, offering improved computational efficiency and flexibility. 
We demonstrate the effectiveness of our approach through comprehensive numerical examples, showcasing its potential for practical applications.}

\end{abstract}





\end{frontmatter}

\section{Introduction}
\label{sec:introduction}
A switched system is defined by a family of subsystems, i.e., modes, and a switching rule, i.e., switching sequence, 
that orchestrates the transition between each mode, see, e.g., \cite{springer, liberzon2005switched,fu2023dynamic}. 
Each instant that separates two modes is called switching instant or switching time.
These systems are commonly used in practical situations where the system's dynamics can change depending on either 
internal or external factors \cite{zhu2015optimal,katayama2023model}.

In recent years, the optimization of switched systems has been studied extensively \cite{zhu2015optimal}. 
This line of research can be divided into two branches: optimization of autonomous systems and optimal control of non-autonomous systems.
In the case of autonomous systems \cite{zhu2015optimal,egerstedt2006transition, stellato2016optimal, stellato2017second, example}, 
the optimization focuses on obtaining the optimal switching sequence, which often involves optimizing 
the switching times of a pre-specified sequence.
In this paper, we focus on the non-autonomous branch.

\rev{
A widely adopted technique is the time-scaling transformation \cite{lee1997control,loxton2014switching,teo2021applied} in conjunction
with control parameterization, wherein a free-time optimal control problem is reformulated into a fixed-time domain. 
This reformulation is achieved by introducing a new time variable that aligns switching instants with fixed nodes in the scaled
timeline. 
The control input is then approximated as a piecewise constant or polynomial function over each subinterval. 
Although this method facilitates tractable optimization, it traditionally requires that control switching occurs concurrently
with subsystem switching, a condition that may preclude optimality in some settings. 
To address this limitation, \cite{zhu2022sequential} proposed an enhanced computational approach that decouples control input switching instants from
those of subsystem transitions, thereby expanding the admissible class of control strategies and improving overall performance.
However, this approach may be limited in practical applications where there is no need to switch the control input at different times than the switching instants.
}

\rev{
Another popular approach is the two-stage optimization algorithm,
which was first introduced in \cite{opt_contr_swi} and it consists of different optimization stages.
The lower level assumes the switching sequence to be fixed and solves an optimization problem, optimizing the control input and the switching instants. 
Afterwards, the upper level minimizes the cost function associated with the switching sequence by solving another optimization problem.
An alternation between these two stages fully determines the optimal solution. 
However, the upper level of this two-stage optimization approach may be subject to a combinatorial explosion.
This is because the number of possible switching sequences grows exponentially with the number of modes.
To mitigate the challenges posed by this combinatorial explosion, several approaches
\cite{egerstedt2006transition, axelsson2005transition, axelsson2008gradient, wardi2012algorithm} focus on fixing the switching sequence and optimizing only the
control inputs and the switching instants. This simplification reduces
computational complexity while still yielding practical solutions.
Recently, \cite{zhu2024new} introduced a two-layer optimization scheme that decouples the optimization of subsystem dynamics and control switches.
Likewise, \cite{zhang2019data} demonstrates that data-driven adaptive control methods can curb the exponential growth of possible switchings
in high-dimensional linear uncertain switched systems by leveraging online estimation and approximation techniques.
}

\rev{
Alternative strategies involve mixed-integer nonlinear programming, which encodes switching decisions via binary variables. 
While globally optimal in principle, these formulations are computationally expensive.
To circumvent this, the embedding transformation methods have been proposed \cite{bengea2005optimal,wei2007applications}.
The embedding transformation aims to convert switched systems into a more extensive set of continuous systems. 
\cite{abudia2020switched} demonstrated that this relaxation allows the problem to be recast as a
continuous optimal control problem while still respecting practical constraints
such as minimum dwell time.
Leveraging these transformations, conventional optimization techniques can be employed without requiring any prior assumptions 
about either the number of switchings or the switching sequence.
Nevertheless, finding a solution to the embedded optimal control problem may lead to a suboptimal solution.
}

\rev{
When the mode sequence is held fixed, gradient-based methods become particularly effective. 
Building on \cite{wardi2015switched}, which provided explicit
expressions for the cost gradient with respect to switching instants, such
methods iteratively adjust switching times via steepest descent. 
\cite{liu2014switching,liu2018dynamic} extended this approach to include the optimization of control inputs.
In \cite{liu2014switching}, the Authors present a rigorous method that proves the existence and form of state
derivatives with respect to both parameters and switching times, even in the presence of delays.
This method is also employed on switched time-delay systems with state-dependent switching conditions in \cite{liu2018dynamic}.
To improve convergence rates and robustness, in \cite{stellato2016optimal,stellato2017second} the Authors developed a Newton-type 
algorithm that employs second-order derivatives computed via linearization over a discretized time grid.
This approach allows for the identification of shared terms between the cost
function and its derivatives, which can be computed offline.
}

\rev{
From a practical standpoint, the Constrained Linear Quadratic Regulator (CLQR) extends
the classical Linear Quadratic Regulator (LQR) to address optimal control of linear
systems while incorporating input constraints \cite{scokaert1998constrained}, making it a widely adopted
technique in control applications \cite{DEPERSIS2021109548,he2024approximate}.
The capacity of LQR and CLQR to provide nominal, closed-loop stability for unconstrained
and input-constrained linear systems, respectively, represents one of the most compelling advantages of these techniques.
Concerning continuous-time switched LQR problems, \cite{wu2018optimal} presents a method that leverages the embedding transformation method.
}

\rev{
In this work, we propose a novel method for solving the Continuous-Time switched Constrained Linear Quadratic Regulator (CLQR) problem. 
Building upon time-scaling techniques, our method introduces a unified framework that simultaneously optimizes both the control inputs and the mode durations in a
single optimization stage, without relying on other methods such as two-stage optimization or embedding transformations.
A key feature of our method lies in the identification and exploitation of shared terms within the analytical expressions of both the cost function and its gradient. 
These terms can be efficiently precomputed offline, thereby substantially reducing the computational complexity during the optimization process, 
akin to the approach in \cite{stellato2016optimal,stellato2017second}. 
By addressing the regulation problem, we incorporate control inputs directly into the optimization framework, 
overcoming challenges associated with deriving the analytical gradient of the cost function. 
}

The main contributions of this paper are summarized as follows:
\begin{Aenumerate}
    \item We propose a novel unified framework for solving the Continuous-Time CLQR problem for switched systems. 
    Our approach simultaneously optimizes both the control inputs and the mode durations in a single optimization stage, 
    eliminating the need for multi-stage methods or embedding transformations. 
    \item By leveraging matrix exponential, we reformulate the CLQR problem to explicitly highlight 
    shared terms in the analytical expressions of the cost function and its gradient. 
    These shared terms can be precomputed offline, significantly reducing the computational complexity during optimization. 
    \item \rev{We demonstrate the effectiveness of our method through numerical examples that showcase its performance in solving the Continuous-Time CLQR problem for switched systems.}
\end{Aenumerate}
The remaining of this paper is organized as follows.
In \cref{sec:problem_definition}, we present the switched systems continuous-time CLQR problem we want to address.
In \cref{sec:solution_approach}, we introduce the time grid partitioning and the parameterization of the dynamics through the duration of each partition. Moreover, we motivate the matrix exponential discretization and the formulation that explicits the shared terms in the analytical expressions of the cost function and its gradient.
In \cref{sec:numerical_solution}, we show the numerical implementation of our approach.
\Cref{sec:limitations} discusses the potential limitations of our method, including local minima and the combinatorial explosion of the switching sequence.
Then, in \cref{sec:examples}, we validate our method with some numerical examples.

\section{Problem Definition}	\label{sec:problem_definition}
Let $x(t) \in \R^n$ be the state, $u(t) \in \R^m$ be the control input, ${}^i A \in \R^{n \times n}$ \rev{be} the dynamic matrix, 
and ${}^i B \in \R^{n \times m}$ be the input matrix that represents the linear evolution of the generic active mode $i$.
Recalling \cite{liberzon2005switched}, a non-autonomous switched continuous-time linear system is defined as 
\begin{equation}	\label{eq:switched_sys}
    \begin{split}
        &\dot{x}(t) = {}^i A x(t) + {}^i B u(t) \\
        &\textnormal{with} \quad i \in \cI \triangleq \{1,\,\ldots,\,N_\textnormal{I}\}
    \end{split}
\end{equation}
where $\cI$ is the set of indices of modes, $t\in[t_0,\,t_{\textnormal{f}}]$ is the time domain, 
and $N_\textnormal{I}\in\N$ denotes the total number of modes.

Choosing the correct switching sequence in addition to the control input is crucial for efficiently controlling a switched system.
The switching sequence in $[t_0,\,t_{\textnormal{f}}]$ regulates the sequence of active modes and is defined as 
$\sigma(\cdot)~:~[t_0,\,t_{\textnormal{f}}]~\to~\cI$. 
The switching sequence is a piecewise constant function, and we assume that it has a finite number of discontinuities, 
called switching instants, which are the time instants when the transition from one mode to another occurs.
We define with $\tau = [\tau_1,\,\tau_2,\,\ldots,\, \tau_N]^{\top}$ the sequence of the switching instants, 
where $N$ indicates the total number of switches.
We enforce that $\tau \in \cT(t_0,\,t_{\textnormal{f}})$, where the set $\cT(t_0,\,t_{\textnormal{f}})$ is defined as
\begin{equation}	\label{T}
    \cT(t_0,\, t_{\textnormal{f}}) \triangleq \{\tau = [\tau_1,\, \ldots,\, \tau_N]^{\top} \,|\,t_0 \leq \tau_1 \leq \ldots \leq \tau_{N} \leq t_{\textnormal{f}} \}.
\end{equation}
Indicating with $i$ the generic mode and with $k= 1, \ldots, N$ the time instant when the transition occurs, 
then $i_k$ denotes that the $i$-th dynamics remains active along the time 
interval $\left[\tau_k,\, \tau_{k+1}\right]$.
\rev{
An implicit outcome of \cref{eq:switched_sys} is that the state \textit{x} remains continuous and only $\dot{x}$ can change abruptly when the switching occurs; thus, we state the following assumption:
\begin{assumption} \label{assumption_1}
    The state of the system is continuous at the switching instants, i.e., $\lim_{t\rightarrow\tau_k^-} x(t) = \lim_{t\rightarrow\tau_k^+} x(t)$ for $k = 1, \ldots, N$.
\end{assumption}
This assumption is a common assumption in the literature on switched systems \cite{liberzon2005switched, stellato2017second}.
}
Solving the Continuous-Time CLQR problem for switched systems involves not only
determining the optimal control input $u$ that minimizes the cost function but
also optimizing the switching sequence $\sigma$. 
The timing of the switching instants plays a pivotal role in identifying the optimal switching sequence. 
As highlighted in \cite{stellato2016optimal, stellato2017second, opt_contr_swi}, the switching sequence 
can be parameterized in terms of the switching instants.
Consequently, optimizing the switching instants inherently optimizes the switching sequence, 
integrating both aspects into a unified optimization framework.
We can then define the optimization problem as follows:
\begin{prbl} \label{problem_1}
    \rev{
    Let us consider a switched system in the form of \cref{eq:switched_sys}, an initial and a final time, namely $t_0$ and $t_\textnormal{f}$, 
    a predefined maximum number of switchings $N$, and a fixed mode sequence with $N_{\textnormal{I}}$ modes. 
    Determine, if it exists, a continuous function $u^* : [t_0, t_\textnormal{f}] \to \R^m$ and a switching instant sequence $\tau^*$ that solves the continuous-time optimization problem $\mathcal{P}^{C}_{x_0}$ defined as
    \begin{mini!}
        {u, \tau}{\frac{1}{2} \int_{t_0}^{t_\textnormal{f}}  \left(x^{\top}(t) \Tilde{Q} x(t) + u^{\top}Ru(t) \right)\,  \textnormal{d}t + \frac{1}{2} x^\top(T) \Tilde{E} x(T) \protect\label{eq:obj_finite_dim_prob}}{\label{eq:finite_dim_prob}}{}
        \addConstraint{\dot{x}(t)}{= {}^i A x(t) + {}^i B u(t) {\quad \forall t \in \left[ \tau_{k-1}, \, \tau_k \right]} {\quad i = k - N_{\textnormal{I}} \left \lfloor \frac{k}{N_{\textnormal{I}}} \right \rfloor} {\quad k = 0, \, \ldots, \, N,  \protect\label{eq:c1_finite_dim_prob}}}
        \addConstraint{x(0)}{= x_0, \protect\label{eq:c2_finite_dim_prob}}
        \addConstraint{u(t)}{\in \mathcal{U},  \protect\label{eq:c3_finite_dim_prob}}
        \addConstraint{\tau}{\in \cT(t_0,t_{\textnormal{f}}),
        \protect\label{eq:c4_finite_dim_prob}}
    \end{mini!}
where $\mathcal{U} \triangleq \{u \in \R^m \, | \, u_{\textnormal{min}} \leq u \leq u_{\textnormal{max}} \}$, where $u_{\textnormal{min}}$, $u_{\textnormal{max}} \in \R^m$ are the lower and upper bounds of the control input, respectively.
}
\end{prbl}
The cost function~\cref{eq:obj_finite_dim_prob} minimizes the state vector's distance from the origin and the control input over the time horizon, considering $\Tilde{Q}$ as an $n \times n$ positive semidefinite matrix, and $R$ is an $m \times m$ positive definite matrix.
The final state's displacement from the origin is further penalized by the terminal cost, where $\Tilde{E}$ is an $n \times n$ positive definite matrix.
\Cref{eq:c1_finite_dim_prob} represents the system's dynamics, embedding the switching sequence.
\rev{In \cref{eq:c1_finite_dim_prob}, $i$ is the index of the active mode. 
The modulo operator ensures the fixed mode sequence, where $k$ is the index of the previous
switching instant, and $N_{\textnormal{I}}$ is the number of modes.}
Furthermore, considering the constraint \cref{eq:c4_finite_dim_prob} is included to ensure adherence to the conditions outlined in \cref{T}. 

\section{Solution Proposed}	
\label{sec:solution_approach}	
In this section, we present a novel framework for solving \cref{problem_1} leveraging a time partitioning formulation,
which enables us to compute the cost function \cref{eq:obj_finite_dim_prob} and its gradient analytically.
With this approach, we aid the optimization solver in finding the optimal solution more efficiently.
Moreover, as introduced in \cref{sec:problem_definition}, we parameterize the switching sequence as a function of the partition's duration.
This methodology provides a systematic approach to solving \cref{problem_1},
facilitating the simultaneous optimization of the switching sequence and partition duration.

\subsection{Preliminaries} \label{preliminaries}
    We partition the time horizon $\left[ t_0,\, t_\textnormal{f} \right]$ into a sequence of $N + 1$ intervals,
    with $N$ being the number of switching instants, where each interval is defined 
    by its duration $\delta_k = \tau_{k+1} - \tau_k$, with $\tau_k$ and $\tau_{k+1}$ representing consecutive switching instants. 
    The duration $\delta_k$ corresponds to the time spent in each system mode. 
    Consequently, as per \cref{eq:switched_sys}, the index $i$ of the tuple $\left({}^i A, \,{}^i B \right)$ changes 
    at the boundaries of these intervals, i.e., when transitioning from one partition to the next. 
    To formalize this, we define the following set:
    \rev{
    \begin{equation} \label{Delta}
        \Delta(t_0,\,t_{\textnormal{f}}) \triangleq \left\{ \delta \in \R^{N+1}\,|\, \delta_k \geq 0 \; \textnormal{and} \; \sum_{k=0}^{N} \delta_k = t_{\textnormal{f}} - t_0 \right\},
    \end{equation}
    }
    which ensures that the duration of each partition is non-negative and collectively spans the entire 
    time horizon $\left[ t_{0},\, t_{\textnormal{f}}\right]$.
    \rev{
    It is important to emphasize that our method does not impose a fixed number of switching instants. 
    By treating the duration of each mode, $\delta_k$, as an optimization variable and enforcing the condition $\delta_k \geq 0$, 
    we inherently allow the number of switching instants to vary. 
    Specifically, if the optimal solution yields $\delta_k = 0$, the corresponding mode $i_k$ is effectively excluded 
    from the optimal control strategy.
    }
    Leveraging the partitioning framework, we reformulate \cref{problem_1} to focus on 
    optimizing the partition duration rather than the switching instants. 
    This reformulation enables a unified approach for the simultaneous optimization of both control inputs and 
    partition duration.
    \rev{To streamline the presentation of the method in the subsequent sections, 
    we introduce the following assumption:
    \begin{assumption} \label{assumption_2}
        The control input $u(t)$ is piecewise constant, i.e.,
        \begin{equation}
        \label{eq:piecewise_input}u(t) = u_{k}, \, t \in \left[\tau_{k},\, \tau_{k+1}
        \right], \, k = 0,\, \ldots,\, N_{\textnormal{I}},
    \end{equation}
    where $u_k \in \R^{m}$ is the control input associated with the $k$-th partition.
    \end{assumption}
    This assumption enforces that the control input remains constant within each partition, 
    which simplifies the state evolution analysis.}
    \rev{Recalling \cref{eq:switched_sys,Delta,assumption_2}, the state evolution within each partition, given the initial state $x_k$, 
    is entirely determined by the pair $\left(u_k, \, \delta_k\right)$. 
    These considerations are pivotal for the derivation of the proposed method.}
    Concerning $\sigma(t)$, we implement a predefined switching sequence.
    Let us consider the $N_\textnormal{I}$ modes of \cref{eq:switched_sys}. 
    We define the switching sequence as the extended sequence that repeats the canonical sequence, namely,
    $i = 1, \, 2, \, \ldots , \, N_{\textnormal{I}}$, $n_\textnormal{c}$ times, with $n_\textnormal{c} \in \N_{ > 0} $.
    For a more detailed understanding of the structure of the switching sequence, the reader is encouraged 
    to consult \cref{fig:predefined_sequence}.
    \begin{figure}[t]
        \centering
        \begin{tikzpicture}
            \draw[->] (-4.2,1.5) -- (10.2,1.5);

            \draw (-4.0,0.1+1.5) -- (-4.0,-0.1+1.5) node[below] {$t_{0}$};
            \node[above] at (-4+1, 1.6) {${}^{1} A, {}^{1} B$};
            \node[above] at (-4+1, 0.9) {\textcolor{blue}{$A_1, B_1$}};
            \draw (-4.0+2,0.1+1.5) -- (-4.0+2,-0.1+1.5) node[below] {$\tau_{1}$};
            \node[above] at (-4.0+3, 1.6) {${}^{2} A, {}^{2} B$};
            \draw (-4.0+4,0.1+1.5) -- (-4.0+4,-0.1+1.5) node[below] {$\tau_{2}$};
            \node[above] at (-4.0+4.5, 1.6) {$\ldots$};
            \draw (-4.0+5,0.1+1.5) -- (-4.0+5,-0.1+1.5)
                node[below] {$\tau_{N_\textnormal{I}}$};
            \node[above]
                at
                (-4.0+6, 1.6)
                {${}^{N_\textnormal{I}} A, {}^{N_\textnormal{I}} B$};
            \draw (-4.0+7,0.1+1.5) -- (-4.0+7,-0.1+1.5)
                node[below] {$\tau_{N_\textnormal{I}+1}$};
            \node[above] at (-4.0+8, 1.6) {$\ldots$};
            \draw (-4.0+9,0.1+1.5) -- (-4.0+9,-0.1+1.5)
                node[below] {$\tau_{(n_c-1)(N_\textnormal{I}+1)}$};
            \node[above] at (-4.0+9.8, 1.6) {${}^{1} A, {}^{1} B$};
            \draw (-4.0+10.5,0.1+1.5) -- (-4.0+10.5,-0.1+1.5);
            \node[above] at (-4.0+11.3, 1.6) {$\ldots$};
            \draw (-4.0+12,0.1+1.5) -- (-4.0+12,-0.1+1.5) node[below] {$\tau_{N}$};
            \node[above] at (-4+13, 1.6) {${}^{N_{\textnormal{I}}} A, {}^{N_{\textnormal{I}}} B$};
            \node[above] at (-4+13, 0.9) {\textcolor{blue}{$A_{N}, B_{N}$}};
            \draw (-4.0+14,0.1+1.5) -- (-4.0+14,-0.1+1.5)
                node[below] {$t_{\textnormal{f}}$};

            \draw[decorate, decoration={brace, mirror, amplitude=5pt}, yshift=-15pt]
                (-2,1.5) -- (0,1.5)
                node[midway, below=6pt] {$\delta_1$};
            \draw[decorate, decoration={brace, mirror, amplitude=5pt}, yshift=-15pt]
                (1,1.5) -- (3,1.5)
                node[midway, below=6pt] {$\delta_{N_{\textnormal{I}}}$};
        \end{tikzpicture}
        \caption{The figure illustrates the construction process of the predetermined sequence, 
        which involves repeating the canonical sequence by $n_{\textnormal{c}}$ times.}
        \label{fig:predefined_sequence}
    \end{figure}
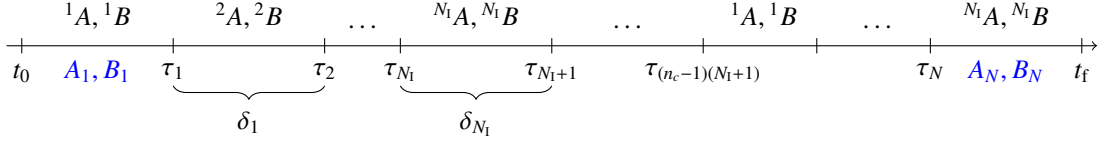
    \rev{This procedure ensures that the optimized switching sequence explores a diverse set of mode combinations, 
    thereby enhancing the solution's robustness and adaptability. 
    While the proposed approach does not entirely resolve the combinatorial explosion challenge, 
    it provides a practical and scalable framework to mitigate its impact \cite{zhu2024new}.
    Notably, the fixed switching sequence condition is relaxed by permitting $\delta_{k} = 0$, 
    which implies $\tau_{k}= \tau_{k+1}$. 
    This allows the switching sequence to be dynamically adjusted by directly transitioning from $(i-1)_{k-1}$ to $(i+1)_{k+1}$,
    effectively skipping intermediate modes.}
    \rev{To enhance clarity and maintain consistent notation, we associate each vector and matrix quantity with the corresponding time partition. 
    For instance, the vector $x_{k}$ and the matrix $A_{k}$ are linked to the $k$-th partition of duration $\delta_{k}$. 
    Referring to \cref{fig:predefined_sequence}, there exists a unique correspondence between the partition index $k$ and the active mode $i$. 
    Specifically, for cases where $N > N_{\textnormal{I}}$, the active mode $i$ can be determined as $i = k - N_{\textnormal{I}} \left\lfloor \frac{k}{N_{\textnormal{I}}}\right\rfloor$. 
    This systematic association simplifies the interpretation of the switching sequence and its relationship with the time partitions. }   
    Moreover, for a given switching instant $\tau_k$, we use the following state notation $x(\tau_k) = x_k$.
    \rev{Based on the time partitioning previously introduced and on \cref{assumption_2}, we can rewrite \cref{eq:finite_dim_prob} highlighting
    the pair $\left(u_k, \, \delta_k\right)$ as the optimization variables, as follows:
    
    
    \begin{mini!}
        {\substack{u_0, \ldots, u_N \\ \delta_0, \ldots, \delta_N}}{\frac{1}{2} \sum_{k=0}^{N} \left( \int_{0}^{\delta_{k}} \left( x(t)^\top \tilde{Q} x(t)+ u_{k}^\top R u_k \right) \textnormal{d}t \right) + \frac{1}{2} x(t_{\textnormal{f}})^\top \tilde{E} x(t_{\textnormal{f}}) \protect\label{eq:obj_finite_dim_prob_2}}{\label{eq:finite_dim_prob_2}}{}
        \addConstraint{\dot{x}(t)}{= A x(t) + B u_k,} {\quad \forall t \in [0, \delta_k],}{\quad \forall k = 0, \ldots, N,  \protect\label{eq:c1_finite_dim_prob_2}}
        \addConstraint{x(0)}{= x_0,  \protect\label{eq:c2_finite_dim_prob_2}}
        \addConstraint{u_{\textnormal{min}}}{\leq u_k \leq u_{\textnormal{max}}}{\quad \forall k = 0, \ldots, N,  \protect\label{eq:c3_finite_dim_prob_2}}
        \addConstraint{\delta_k}{\geq 0}{\quad \forall k = 0, \ldots, N, \protect\label{eq:c4_finite_dim_prob_2}}
        \addConstraint{\sum_{k=0}^N \delta_k}{= t_{\textnormal{f}} - t_0,  \protect\label{eq:c5_finite_dim_prob_2}}
    \end{mini!}
    where \cref{eq:c4_finite_dim_prob_2,eq:c5_finite_dim_prob_2} ensure that the duration of all partitions fulfill \cref{Delta}.
    This formulation allows simultaneous optimization of the control inputs, the duration of each mode, 
    and the adjustment of the switching sequence. 
    By leveraging a time-scaling-like transformation \cite{loxton2014switching,lin2013optimal}, 
    we ensure computational efficiency and flexibility in handling complex switching dynamics.
    }
    
    
    \subsection{State-Transition Matrix}	
    \label{state_transition}
    Here, we introduce the matrix exponential parameterization used to discretize the continuous-time CLQR. 
    
    Let us consider any linear system for which the state evolution is described as in \cite{bakshi2020modern}.
    Recalling \cref{eq:switched_sys,assumption_2}, we choose two generic switching instants, i.e.,  
    $\tau_k $ and $\tau_\ell$, where $ \tau_k < \tau_\ell $, in accordance with the time partitioning outlined in \cref{preliminaries}.
    Then, we can write the state transition between those switching instants.
    \begin{equation} \label{eq:lagrange_formula}
        x_{\ell}=\prod_{q=k}^{\ell-1} e^{A_q\delta_q} x_k + \sum_{q=k}^{\ell-1} \left(\prod_{s=q+1}^{\ell-1} e^{A_s\delta_s} \right) \Phi_{\textnormal{F}_q} u_q,
    \end{equation}
    where
    \begin{equation} \label{eq:phi_f_i}
        \Phi_{\textnormal{F}_q} = \int_{0}^{\delta_{q}}e^{A_q(\delta_{q}-t)}B_q \, \textnormal{d}t \in \R^{n \times m}.
    \end{equation}
    One can decompose \cref{eq:lagrange_formula} into two terms: one representing the autonomous evolution and the other representing the forced evolution, namely $\Phi_\textnormal{A}$ and $\Phi_\textnormal{F}$.
    \begin{align}
        \rev{\Phi_\textnormal{A}(\tau_\ell, \, \tau_k)} \; & \rev{=\prod_{q=k}^{\ell-1} e^{A_q \delta_q} \in \R^{n \times n}, \label{phi_a} }\\
        \Phi_\textnormal{F}(\tau_\ell, \, \tau_k) &= \sum_{q=k}^{\ell-1}\left(\prod_{s=q+1}^{\ell-1} e^{A_s \delta_s} \right) \Phi_{F_q} u_q \in \R^{n \times 1}, \label{phi_f}
    \end{align}
    \rev{where the parentheses indicate the time intervals of the state evolution.
    For the sake of clarity, with this notation, we can represent the state evolution across
    multiple time partitions.}
    To encapsulate the complete information of the state coming from \cref{eq:lagrange_formula}, 
    we define the state-transition matrix $\Phi(\tau_\ell, \, \tau_k) \in \R^{(n+1)\times(n+1)}$ as follows
    \begin{equation} \label{eq:transition_matrix}
        \Phi(\tau_\ell, \, \tau_k) = 
        \begin{bmatrix}
            \Phi_\textnormal{A}(\tau_\ell, \, \tau_k) & \Phi_\textnormal{F}(\tau_\ell, \, \tau_k) \\
            \underline{0} & 1
        \end{bmatrix}.
    \end{equation}
    In order to keep a compact notation, we augment the state as $\overline{x}(t) = \left[x(t)^\top,\,1\right]^{\top}$.
    This leads us to propagate the augmented state as follows
    \begin{equation} \label{state}
        \overline{x}_{\ell} = \Phi(\tau_\ell, \, \tau_k) \overline{x}_k.
    \end{equation}
    In a different perspective, we can consider \crefrange{eq:lagrange_formula}{eq:transition_matrix} to be a discretization of \cref{eq:switched_sys} on the time partitioning introduced in \cref{preliminaries}.
    Thus, we can compute the equivalent discrete state evolution as: $x_{k+1} = \Phi_\textnormal{A}(\tau_{k+1}, \tau_k) x_k + \Phi_\textnormal{F}(\tau_{k+1}, \tau_k)$.
    \rev{The state-transition matrix is explicitly parameterized by the duration of each partition, $\delta_{k}$, 
    which plays a pivotal role in the optimization process.
    This parameterization enables the analytical computation of both the cost function
    and its gradient with respect to the control inputs and partitions duration.
    Furthermore, it facilitates the dynamic adjustment of the switching sequence by optimizing
    the duration of each partition. 
    This unified approach allows for the simultaneous optimization of control inputs and partitions duration, 
    streamlining the overall optimization process into a cohesive framework.} 

    The state-transition matrix defined in \cref{eq:transition_matrix} satisfies the two following lemmas.
    \begin{lemma} \label{lemma_1}
        Given two switching instants $\tau_a$ and $\tau_\ell$ with $\tau_a \leq \tau_\ell$, it holds that
        \begin{equation} 
            \label{composition_rule}
            \Phi(\tau_\ell, \, \tau_a) =  \Phi(\tau_\ell, \, \tau_k) \Phi(\tau_k, \, \tau_a)\,.
        \end{equation}
    \end{lemma}	 
    \begin{proof}
        The proof is in \ref{proof:lemma_1_2}.
    \end{proof}

    \begin{lemma} \label{lemma_2}
        Given two switching instants $\tau_a$ and $\tau_{k+1}$ with $\tau_a \leq \tau_{k+1}$, the derivatives of $\Phi(\tau_{k+1}, \, \tau_a)$ 
        w.r.t. the switching interval, i.e., $\delta_k$,  and the control input, i.e., $u_{kj}$, namely, the $j$-th component 
        of the control input $u_k$, can be written as:
        \begin{equation} \label{lemma_1_eq}
            \frac{\partial \Phi(\tau_{k+1}, \, \tau_a)}{\partial \delta_k} =  M_k \Phi(\tau_{k+1}, \, \tau_a),
        \end{equation}
        where 
        \begin{equation} \label{M_matrix}
            M_k = 
            \begin{bmatrix}
                A_k & B_k u_k(\tau_{k+1}) \\
                \underline{0} & 0
            \end{bmatrix}
            \in \R^{(n+1)\times(n+1)}\,;
        \end{equation}
        and
        \begin{equation} \label{lemma_2_eq}
            \frac{\partial \Phi(\tau_{k+1}, \, \tau_a)}{\partial u_{kj}} =  H_{kj} \Phi(\tau_{k+1}, \, \tau_a),
        \end{equation}
        where
        \rev{
        \begin{equation} \label{H_matrix}
            H_{kj} =
            \begin{bmatrix}
                \underline{0} & \Phi_{F_k}^j \\
                \underline{0} & 0
            \end{bmatrix}
            \in \R^{(n+1)\times(n+1)}.
        \end{equation}
        with $\Phi_{F_k}^j$ being the $j$-th column of the matrix $\Phi_{F_k}$ introduced in \cref{eq:phi_f_i}.}
    \end{lemma}
    
    \begin{proof}
        The proof of Lemma \ref{lemma_2} is in \ref{proof:lemma_1_2}.
    \end{proof}
        
    The state-transition matrix allows us to propose the following definitions, which are essential for establishing the main results of this study.
    \begin{definition} \label{definition_1}
        Given the switching instant $\tau_k\in[t_0,\, t_\textnormal{f}]$, we define the symmetric and positive semidefinite matrix $P_k \in \R^{(n+1)\times (n+1)}$ as
        \begin{equation} \label{P_matrix}
            P_k \triangleq \frac{1}{2} \int_{\tau_k}^{t_{\textnormal{f}}} \Phi(t, \tau_k)^\top Q \Phi(t, \tau_k) dt,
        \end{equation}
        where $\Phi(t, \tau_k)$ is the state-transition matrix \cref{eq:transition_matrix}, and $Q = \textnormal{blkdiag}(\Tilde{Q},\,0)$.
        We also define the symmetric and positive semidefinite matrix $F_k \in \R^{(n+1)\times (n+1)}$~as
        \begin{equation} \label{F_matrix}
            F_k \triangleq \frac{1}{2} \Phi(t_{\textnormal{f}}, \tau_k)^\top E \Phi(t_{\textnormal{f}}, \tau_k),
        \end{equation}
        where $E = \textnormal{blkdiag}(\Tilde{E},\,0)$.
        Finally, we define the sum of these two matrices as
        \begin{equation} \label{S_matrix}
            S_k \triangleq P_k + F_k, \quad k = 0,\, \ldots,\, N+1. 
        \end{equation}
        We also define the scalar term $G_k \in \R$ as
        \begin{align} 
            G_k &\triangleq \frac{1}{2} \int_{\tau_k}^{t_{\textnormal{f}}} u(t)^\top R u(t) \, \textnormal{d}t \label{G} \\
            &= \frac{1}{2} \sum_{q=k}^{N} (u_q^\top R u_q) \delta_q. \label{second_term}
        \end{align}
    \end{definition}
    Note that \eqref{second_term} is obtained recalling \cref{assumption_2}, i.e., considering the input as a piecewise constant function.
    
    \begin{definition} \label{definition_2}
        Given the matrix $H_{kj}$ \eqref{H_matrix}, we define the matrix $D_{kj} \in \R^{(n+1) \times (n+1)}$ as
        \begin{equation} \label{D_matrix}
            D_{kj} \triangleq \frac{1}{2} \int_{\tau_k}^{\tau_{k+1}} H_{kj}^\top(t, \tau_k)Q \Phi(t,\tau_k) + \Phi^\top(t,\tau_k)QH_{kj}(t, \tau_k) \, \textnormal{d}t,
            \quad k = 0,\, \cdots,\, N, \quad j=1,\, \ldots,\, m,
        \end{equation}
        where \textit{j} denotes each component of the \textit{k}-th control input and $\Phi(t, \tau_k)$ is the state-transition matrix \cref{eq:transition_matrix}.
        Recalling \cref{lemma_2,definition_1}, we define the matrix $C_k \in \R^{(n+1) \times (n+1)}$ as
        \begin{equation} \label{C_matrix}
            C_k \triangleq \frac{1}{2} Q + M_k^\top S_{k+1} + S_{k+1} M_k, \quad k = 0,\, \ldots,\, N.
        \end{equation}
        Finally, given matrices $H_{kj}$, i.e., \eqref{H_matrix} and $S_k$, we define the matrix $N_{kj} \in \R^{(n+1) \times (n+1)}$ as
        \begin{equation} \label{N_matrix}
            N_{kj} \triangleq H_{kj}^\top S_{k+1} + S_{k+1} H_{kj}, \quad k = 0,\, \cdots,\, N, \quad j=1,\, \ldots,\, m.
        \end{equation}
    \end{definition}
    
    For the sake of clarity, \cref{definition_1,definition_2} resemble the matrices in \cite{stellato2017second}. 
    However, it is worth highlighting that here, the matrices depend on both $u$ and $\delta$.

    \subsection{Cost Function and Gradient}	\label{sec:cost_function_and_gradient}
    Using the state-transition matrix discretization and the shared matrices presented in \cref{definition_1,definition_2}, 
    we can reformulate $J(u,\,\delta)$ \cref{eq:obj_finite_dim_prob_2} and its gradient, i.e., $\nabla J(u, \delta)$.
    \rev{This reformulation enables the analytical computation of both the cost function
    and its gradient, offering a significant advantage over numerical methods that depend on finite differences. 
    By leveraging shared terms, this approach not only enhances computational efficiency but also reduces the overall 
    complexity of the optimization process. 
    Consequently, it alleviates the computational burden on the solver, facilitating a faster and more accurate determination 
    of the optimal solution.}
    In the following, \cref{th_J,th_dJ}, represent the main contribution of this work.  
    \begin{theorem}	
        \label{th_J}
        Let us recall \cref{definition_1} and the augmented initial state $\overline{x}_0$, the cost function \cref{eq:obj_finite_dim_prob_2} can be rewritten as
        \begin{equation} \label{eq:J}
            J(u,\,\delta) = G_0 + \overline{x}_0^{\top} S_0 \overline{x}_0.
        \end{equation}
    \end{theorem}
    
    \begin{proof}
        The proof is in \ref{proof:th_1}.
    \end{proof}
    Using this theorem, we can derive the following corollary.
    \begin{corollary} \label{corollary_1}
        If we consider an autonomous switched linear system, namely, the control matrix $B_k$ is always zero, then the cost function $J(\delta)$ reduces to the form
        \begin{equation} \label{eq:J_aut}
            J(\delta) = \overline{x}_0^{\top} S_0 \overline{x}_0,
        \end{equation}
        where the matrix $S_0$ still follows \cref{definition_1}.
    \end{corollary} 
    
    \begin{theorem}	\label{th_dJ}
        Let us recall \cref{definition_1,definition_2}, the gradient $\nabla J(u, \delta) = \left[\frac{\partial J(u,\,\delta)}{\partial u} \, \frac{\partial J(u,\,\delta)}{\partial \delta}\right]^\top$ can be computed as
        \begin{equation} \label{eq:dJ_u}
            \frac{\partial J(u,\,\delta)}{\partial u_{kj}} = u_{kj} R_{jj} \delta_k + \overline{x}_{k}^{\top} D_{kj} \overline{x}_{k} + \overline{x}_{k+1}^{\top} N_{kj} \overline{x}_{k+1}, \quad k = 0,\, \cdots,\, N, \quad j=1,\, \ldots,\, m,
        \end{equation}
        and
        \begin{equation} \label{eq:dJ_delta}
            \frac{\partial J(u,\,\delta)}{\partial \delta_k} = \frac{1}{2} u_k^\top R u_k + \overline{x}_{k+1}^{\top} C_k \overline{x}_{k+1}\,, \quad k = 0,\, \cdots,\, N.
        \end{equation}
    \end{theorem}
                
    \begin{proof}
        The proof is in \ref{proof:th_2}.
    \end{proof}
    From \cref{th_dJ}, we can obtain the following corollary.
    \begin{corollary} \label{corollary_2}
        Considering an autonomous switched linear system, namely, the control matrix $B_k$ is always zero, then the gradient of the cost function can be computed only w.r.t. the partitions duration $\delta_k$ as
        \begin{equation} \label{eq:dJ_delta_aut}
            \frac{\partial J(\delta)}{\partial \delta_k} = \overline{x}_{k+1}^{\top} C_k \overline{x}_{k+1}\,, \quad k = 0,\, \cdots,\, N.
        \end{equation}
    \end{corollary}

    \rev{\Cref{th_J} introduces an analytical expression for the cost function, leveraging
    the state-transition matrix and the matrices defined in \cref{definition_1,definition_2}.
    This formulation significantly enhances computational efficiency by utilizing
    shared terms that can be precomputed, reducing the overall computational burden.
    Similarly, \Cref{th_dJ} provides analytical expressions for the gradient of the cost
    function, streamlining the optimization process and enabling precise adjustments
    to the control inputs and partitions duration.}
    \rev{Notably, the proposed methodology is highly versatile, as it accommodates both
    autonomous and non-autonomous switched linear systems, thereby significantly broadening
    its range of applicability. In contrast to the approach presented in
    \cite{stellato2016optimal,stellato2017second}, which computes the Hessian of the Bolza-form cost function
    and is specifically designed for autonomous systems, our method evaluates only
    the cost function and its gradient, substantially reducing the computational burden.
    Moreover, by directly addressing non-autonomous switched systems without relying
    on local linearizations, our approach extends the applicability of constrained
    LQR solutions while maintaining computational efficiency within switched linear
    frameworks.}

\section{Numerical Solution}
\label{sec:numerical_solution}
\rev{In this section, we present the numerical solution of \cref{eq:finite_dim_prob_2} leveraging the results of \cref{sec:cost_function_and_gradient}.
We first describe the offline computations, which include the computation of the cost function and its gradient. 
Then, we perform the optimization of the cost function exploiting the gradient information.

The core idea of the numerical implementation is to compute the autonomous and
forced evolutions of the switched system within each partition. 
In this way, each term becomes a function of the optimization variables, $u$ and $\delta$. 
By then applying the results from \cref{sec:cost_function_and_gradient}, explicit expressions for the cost function
and its gradient are obtained solely in terms of $u$ and $\delta$. 
This formulation enables the subsequent execution of the optimization procedure to determine the optimal values, namely, $u^*$ and $\delta^*$.}


    \subsection{State Propagation and Matrix Exponential} \label{state_prop_exp_mat}
    \begin{algorithm}[t]
    \caption{Compute exponential matrices}
    \label{alg:alg1}
    \begin{algorithmic}[1] 
        \STATE \textbf{function} \textsc{MatExpProp} 
        \STATE \hspace{0.5cm} $L_{k} \gets \textnormal{Eq. \eqref{eq:exp_mat}}$ 
        \STATE \hspace{0.5cm} $\A(\delta_{k}), \F(\delta_{k}, u_{k}) \gets \textnormal{Eq. \eqref{eq:exp_mat_results}}$ 
        \STATE \hspace{0.5cm} $x_{k+1}\gets \A(\delta_{k}) x_{k} + \F(\delta_{k}, u_{k})$ 
        \STATE \hspace{0.5cm} \textbf{for}$j = 1,\, \ldots,\, m \textbf{ do}$ 
        \STATE \hspace{1.0cm} $H_{kj}\gets \mathrm{col}_{j}(L_{k}^{(2)})$ 
        \STATE \hspace{0.5cm} \textbf{end for} 
        \STATE \hspace{0.5cm} \textbf{return} $x_{k}$ 
        \STATE \hspace{0.5cm} \textbf{return} $\A(\delta_{k}),\, \F(\delta_{k}, u_{k})$
        \STATE \hspace{0.5cm} \textbf{return} $H_{kj},\, \quad j=1,\, \ldots,\, m$ 
        \STATE \textbf{end function}
    \end{algorithmic}
    \label{alg1}
\end{algorithm}
    \rev{In the numerical implementation, the solution of the ODE repeating each mode can be computed either through
    standard numerical integrators, such as Forward Euler or Runge-Kutta, or via the matrix exponential. 
    The latter approach, computed as described in \cite{van1978computing},
    is significantly more efficient, as it yields the solution in a single step rather
    than through iterative methods.
    \begin{definition} \label{definition_3}
        As outlined in \cref{phi_a,phi_f}, for each partition, we compute both the autonomous evolution, given by 
        \begin{equation} \label{phi_aut_i}
            \A(\delta_k) = e^{A_k \delta_k} \quad k = 0,\, \ldots,\, N
        \end{equation}
        as well as the forced evolution, expressed as 
        \begin{equation} \label{phi_forced_i}
            \F(\delta_k, u_k) = \int_{0}^{\delta_{k}}e^{A_k(\delta_{k}-t)}B_k \, \textnormal{d}t\, u_k \quad k = 0,\, \ldots,\, N.
        \end{equation}
    \end{definition}
    As shown in \cite{van1978computing}, we can compute the matrix exponential of a block matrix in a single step.
    The matrix exponential of the block matrix $Z_k$ in \cref{eq:exp_mat} is given by
    \begin{equation} \label{eq:mat}
        Z_k = \begin{bmatrix}A_k&B_k \\ 0&0\end{bmatrix},
    \end{equation}
    where $A_k$ and $B_k$ are the matrices of the system in \cref{eq:finite_dim_prob_2} for the $k$-th partition.
    The matrix exponential of $Z_k$ is computed as follows:
    \begin{equation} \label{eq:exp_mat}
        L_k = e^{Z_k \delta_k} \triangleq \begin{bmatrix} L_k^{(1)}&L_k^{(2)} \\ 0&L_k^{(3)}\end{bmatrix},
    \end{equation}
    where $L_k^{(1)}$ and $L_k^{(2)}$ represent the blocks of the matrix exponential corresponding to the autonomous and forced evolutions, respectively. 
    Specifically, these blocks can be expressed as:
    \begin{equation} \label{eq:exp_mat_results}
        \A(\delta_k) = L_k^{(1)}, \quad \F(\delta_k, u_k) = L_k^{(2)} u_k.
    \end{equation}
    Here, $\A(\delta_k)$ captures the autonomous evolution of the system, while $\F(\delta_k, u_k)$ incorporates the control input $u_k$ to account for the forced evolution. 
    This formulation ensures computational efficiency by leveraging the structure of the matrix exponential.
    To streamline the computation of the matrix $H_{kj}$, we define the function $\mathrm{col}_j(A)$, which extracts the $j$-th column of a matrix $A \in \mathbb{R}^{n \times m}$ as follows:
    $
    \mathrm{col}_j(A) \triangleq (a_{1j}, a_{2j}, \ldots, a_{nj})^\top, \quad \text{where } A = (a_{ij}).
    $
    By applying this function to the matrix $L_k^{(2)}$, the required columns for constructing $H_{kj}$, as defined in \cref{H_matrix}, can be efficiently extracted. This approach minimizes redundant computations and enhances numerical efficiency.
    This ensures that the computation of $H_{kj}$ is streamlined and avoids
    redundant operations.
    \Cref{alg:alg1} provides a detailed procedure for computing the state-transition
    matrices and propagating the state efficiently. 
    This algorithm leverages the matrix exponential to compute both the autonomous and forced evolutions, 
    ensuring computational accuracy and efficiency.
    }
    
    \subsection{Compute Cost Function and Gradient} \label{compute_J_grad}
    To compute the cost function and its gradient, \cref{th_J,th_dJ} are utilized leveraging the shared computations in \cref{definition_1,definition_2}. 
    \cref{alg:alg2} integrates these operations. 
    Although $S_k$ is defined in \cref{S_matrix}, it can be efficiently computed using a recursive approach as outlined in \cref{prep}.

    \begin{proposition} \label{prep}
        The matrix $S_k$ in \eqref{S_matrix} with $k = 0,\, \ldots,\, N$ can be computed satisfying the following recursive equation:
        \begin{align}
            &S_N = \frac{1}{2} E \label{SN}\\
            &S_k = S_k^\Sigma + \Phi(\tau_{k+1},\, \tau_k)^\top S_{k+1} \Phi(\tau_{k+1},\, \tau_k) \label{Si}\\
            &\textnormal{with } S_k^\Sigma = \frac{1}{2} \int_{0}^{\delta_{k}} \Phi(\eta + \tau_k,\,\tau_k)^\top Q \Phi(\eta + \tau_k,\,\tau_k) \, \textnormal{d}\eta.
        \end{align}
    \end{proposition}
    \begin{proof}
    The proof is in \ref{proof_prep}.
    \end{proof}
    As shown in \cref{th_J,prep}, the evolution of the system across each partition of the time interval $\left[t_0, \, t_\textnormal{f} \right]$ is inherently embedded in the cost function formulation. 
    By leveraging matrix exponentials, the need for numerical integration is eliminated, significantly enhancing the computational efficiency of the method. 
    This approach ensures that the cost function and its gradient can be computed with minimal overhead, enabling faster optimization.
    Finally, the optimization problem is solved to find the optimal values of $u$ and $\delta$ that minimize the cost function.
    \begin{algorithm}[h]
        \caption{Compute $J(u,\,\delta)$ and $\nabla J(u,\,\delta)$}\label{alg:alg2}
        \begin{algorithmic}[1]
            \STATE \textbf{function} \textsc{CostFunctionAndGradient}
            \STATE \hspace{0.5cm} Offline computations:
            \STATE \hspace{0.5cm} \textbf{for} $k = 0,\, \ldots,\, N \textbf{ do}$
            \STATE \hspace{1.0cm} $x_k,\, \A(\delta_k), \F(\delta_k, u_k),\, H_{kj} \gets \textsc{MatExpProp}$
            \STATE \hspace{1.0cm} $G_{k} \gets \textnormal{Eq. \eqref{G}}$
            \STATE \hspace{1.0cm} $S_{k} \gets \textsc{ComputeS}$
            \STATE \hspace{1.0cm} $C_{k} \gets \textnormal{Eq. \eqref{C_matrix}}$
            \STATE \hspace{1.0cm} \textbf{for} $j = 0,\, \ldots,\, m \textbf{ do}$
            \STATE \hspace{1.5cm} $D_{kj} \gets \textnormal{Eq. \eqref{D_matrix}}$
            \STATE \hspace{1.5cm} $N_{kj} \gets \textnormal{Eq. \eqref{N_matrix}}$
            \STATE \hspace{1.0cm} \textbf{end for}
            \STATE \hspace{0.5cm} \textbf{end for}
            \STATE \hspace{0.5cm} Compute $J(u,\,\delta)$ and $\nabla J(u,\,\delta)$
            \STATE \hspace{1.0cm} $J(u,\, \delta) \gets \textnormal{Eq. \cref{eq:J}}$
            \STATE \hspace{1.0cm} $\frac{\partial J(u,\,\delta)}{\partial u} \gets \textnormal{Eq. \cref{eq:dJ_u}}$
            \STATE \hspace{1.0cm} $\frac{\partial J(u,\,\delta)}{\partial \delta} \gets \textnormal{Eq. \cref{eq:dJ_delta}}$
            \STATE \hspace{0.5cm} \textbf{return} $J(u,\,\delta)$
            \STATE \hspace{0.5cm} \textbf{return} $\nabla J(u,\,\delta)$
            \STATE \textbf{end function}
        \end{algorithmic}
        \label{alg2}
    \end{algorithm}

\rev{\section{Limitations}
\label{sec:limitations}
In the proposed method, the assumption of a fixed switching signal is relaxed by allowing repetitions of a
canonical sequence, defined as $(1, 2, \ldots, N_{\textnormal{I}})$, with the flexibility to assign zero duration to individual modes. 
This approach broadens the set of admissible switching sequences, but the required number of repetitions to sufficiently cover
a desired set of sequences remains a critical consideration. 
For instance, if the optimal sequence is $(N_{\textnormal{I}}, N_{\textnormal{I}}-1, \ldots, 2, 1)$, achieving this requires $n_{\textnormal{c}}$ to
equal $N_{\textnormal{I}}$, resulting in many modes having zero duration. 
This highlights the sensitivity of the method's performance not only to the choice of the canonical sequence
but also to the number of repetitions $n_{\textnormal{c}}$.
While the framework supports any canonical sequence, including randomized ones that may be advantageous in certain scenarios,
practical applications benefit significantly from leveraging prior knowledge of the physical system to guide this selection. 
Such knowledge ensures that the switching sequences remain feasible and consistent with system dynamics and constraints, 
while excluding irrelevant canonical sequences can improve efficiency and reduce computational overhead.
On the other hand, if there is no knowledge about the system, the method suffers from the curse of dimensionality,
as the majority of the methods in the literature.

Another important consideration is the optimization landscape. 
Since our approach relies on gradient-based methods, it is inherently susceptible to convergence to local minima. 
This issue becomes particularly pronounced when the number of modes is large, 
necessitating an increasing number of repetitions $n_{\textnormal{c}}$ to allow for diverse switching sequences. 
The primary source of non-convexity in the optimization arises from the dependence of the cost function on the mode's duration
$\delta_{q}$, which appears in the exponent of the matrix exponential (see \cref{phi_a,phi_f}). 
In contrast, the cost function remains linear with respect to the continuous control inputs $u_{q}$, partially mitigating this challenge. 
Nonetheless, the presence of multiple local minima can significantly impact solution quality.
Potential strategies to address this issue include employing multiple initializations of the optimization algorithm
or utilizing a global optimization strategy, such as genetic algorithms or particle swarm optimization,
to explore the solution space more thoroughly \cite{9965621}.
Addressing these limitations is an important direction for future research.

}

\section{Examples}
\label{sec:examples}
To demonstrate the effectiveness of our approach, we implemented \cref{alg1,alg2} in Python using the Casadi library \cite{Andersson2019}. 
The cost function \cref{eq:J} and its first derivatives \cref{eq:dJ_u,eq:dJ_delta} were provided to the IPOPT solver \cite{wachter2006implementation}. 
All simulations were performed on an MSI Prestige 15 equipped with an Intel Core i7 processor and 16 GB of RAM. 

In this section, we present two illustrative examples. 
The first example involves a switched linear system with multiple modes and inputs, 
showcasing the versatility of our approach. 
For both examples, the initial values of $\delta_k$ were initialized as equally spaced within the time interval $\left[t_0, t_\textnormal{f} \right]$. 
This choice provides a uniform starting point for the optimization process, ensuring a fair distribution of switching instants across the time horizon.
    \rev{\subsection{Example 1}
    In this example, we analyze a switched linear system characterized by $N_{\textnormal{I}} = 10$ modes and a control parameter $n_\textnormal{c} = 5$.
    As a result, the system operates over $N = 49$ switching instants, leading to a total of $50$ time partitions.
    The dynamics of the switched system are described, according to \cref{eq:switched_sys}, as follows:
    \[
    \dot{x}(t) = {}^i A x(t) + {}^i B u(t), \quad i = 1, \ldots, 10
    \]
    with matrices ${}^i A$ and ${}^i B$ given by:

    \[
    \begin{array}{cccc}
        {}^1 A = \begin{bmatrix}
    -2.5 & 0.5 & 0.3 \\
    0.4 & -2.0 & 0.6 \\
    0.2 & 0.3 & -1.8
    \end{bmatrix} &
    {}^1 B = \begin{bmatrix}
    1.5 & 0.3 \\
    0.4 & 1.2 \\
    0.2 & 0.8
    \end{bmatrix} &
    {}^2 A = \begin{bmatrix}
    -1.9 & 3.2 & 0.4 \\
    0.3 & -2.1 & 0.5 \\
    0 & 0.6 & -2.3
    \end{bmatrix} &
    {}^2 B = \begin{bmatrix}
    1.2 & 0.5 \\
    0.3 & 0.9 \\
    0.4 & 1.1
    \end{bmatrix} \\[1em]

    {}^3 A = \begin{bmatrix}
    -2.2 & 0 & 0.5 \\
    0.2 & -1.7 & 0.4 \\
    0.3 & 0.2 & -2.0
    \end{bmatrix} &
    {}^3 B = \begin{bmatrix}
    1.0 & 0.4 \\
    0.5 & 1.3 \\
    0.3 & 0.7
    \end{bmatrix} &
    {}^4 A = \begin{bmatrix}
    -1.8 & 0.3 & 0.2 \\
    0.5 & -2.4 & 0 \\
    0.4 & 0 & -2.2
    \end{bmatrix} &
    {}^4 B = \begin{bmatrix}
    1.4 & 0.2 \\
    0.6 & 1.0 \\
    0.1 & 0.9
    \end{bmatrix} \\[1em]

    {}^5 A = \begin{bmatrix}
    -2.0 & 0.4 & 0 \\
    0.3 & -2.2 & 0.2 \\
    0.5 & 0.3 & -1.9
    \end{bmatrix} &
    {}^5 B = \begin{bmatrix}
    1.3 & 0.1 \\
    0.2 & 1.4 \\
    0.5 & 0.6
    \end{bmatrix} &
    {}^6 A = \begin{bmatrix}
    -2.3 & 0.2 & 0.3 \\
    0 & -2.0 & 0.4 \\
    0.2 & 0.5 & -2.1
    \end{bmatrix} &
    {}^6 B = \begin{bmatrix}
    1.1 & 0.3 \\
    0.4 & 1.5 \\
    0.2 & 0.8
    \end{bmatrix} \\[1em]

    {}^7 A = \begin{bmatrix}
    -1.7 & 0.5 & 0.4 \\
    0.2 & -2.5 & 0.3 \\
    1.1 & 0.2 & -2.4
    \end{bmatrix} &
    {}^7 B = \begin{bmatrix}
    1.6 & 0.2 \\
    0.3 & 1.1 \\
    0.4 & 0.7
    \end{bmatrix} &
    {}^8 A = \begin{bmatrix}
    -2.1 & 0.3 & 0.2 \\
    0.4 & -1.9 & 0.5 \\
    0.3 & 0.1 & -2.0
    \end{bmatrix} &
    {}^8 B = \begin{bmatrix}
    1.0 & 0.4 \\
    0.5 & 1.2 \\
    0.3 & 0.9
    \end{bmatrix} \\[1em]

    {}^9 A = \begin{bmatrix}
    -2.4 & 0 & 0.5 \\
    0.2 & -2.3 & 0.3 \\
    0.4 & 0.2 & -1.8
    \end{bmatrix} &
    {}^9 B = \begin{bmatrix}
    1.2 & 0.5 \\
    0.1 & 1.3 \\
    0.6 & 0.8
    \end{bmatrix} &
    {}^{10} A = \begin{bmatrix}
    -1.8 & 0.4 & 0.3 \\
    0.5 & -2.1 & 0.2 \\
    0.2 & 3.1 & -2.2
    \end{bmatrix} &
    {}^{10} B = \begin{bmatrix}
    1.4 & 0.3 \\
    0.2 & 1.0 \\
    0.5 & 0.7
    \end{bmatrix}
    \end{array}
    \]

    Our objective is to steer the system from an initial state of $x_0 = [2,\,-1,\,5]^\top$ to the origin 
    within a time horizon of 2 seconds, i.e., $t_0 = 0$s and $t_\textnormal{f} = 2$s.
    To achieve this, we compute \cref{eq:J}, \cref{eq:dJ_u}, and \cref{eq:dJ_delta} using $\tilde{Q} = 10 I_{n\times n}$ 
    (identity matrix), $R = 10 I_{m\times m}$, and $\tilde{E} = I_{n\times n}$ as the terminal cost. 
    The control input is constrained within the bounds $u_{\textnormal{min}} = -1$ and $u_{\textnormal{max}} = 1$.

    \begin{figure}
        \centering
        \begin{subfigure}[t]{0.45\columnwidth}
            \includegraphics[width=\textwidth, keepaspectratio]{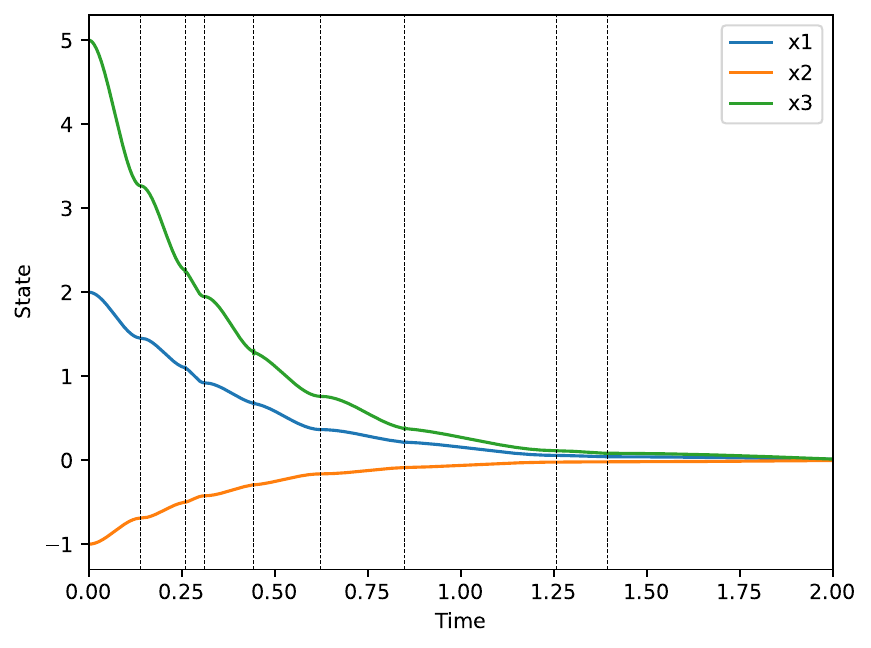}
            \caption{Optimal state evolution.}
            \label{fig:example_1_paper_NAHS_optimal_state}
        \end{subfigure}
        \hfill
        \begin{subfigure}[t]{0.45\columnwidth}
            \includegraphics[width=\textwidth, keepaspectratio]{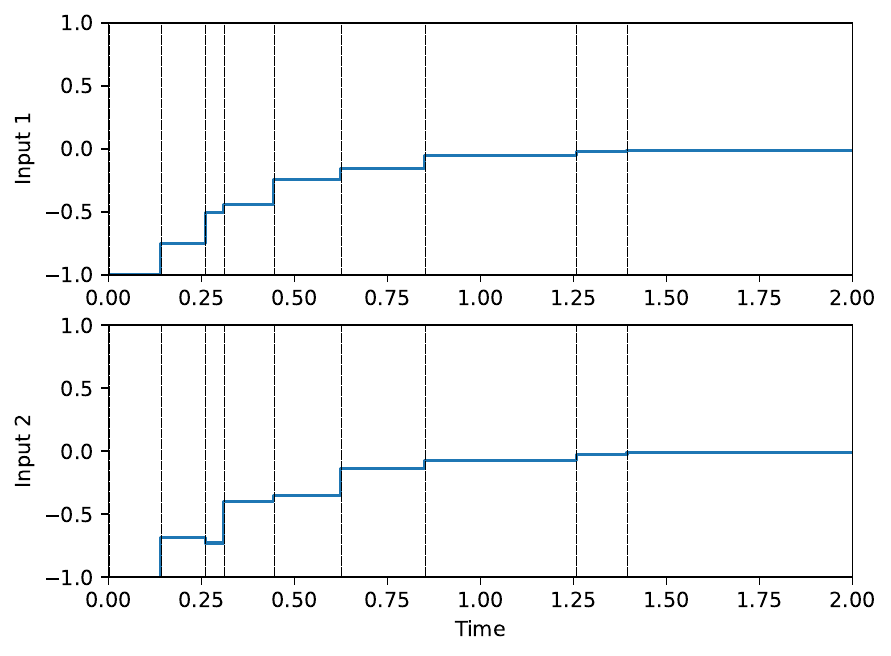}
            \caption{Optimal control input evolution at the optimal switching times.}
            \label{fig:example_1_paper_NAHS_optimal_input}
        \end{subfigure}
        \caption{Results of the optimization: (a) State behavior and (b) Control input evolution.}
        \label{fig:example_1_combined}
    \end{figure}

    The total solving time for the optimization problem is 3.54 seconds, with 3.125 seconds 
    dedicated to offline computations and 0.416 seconds utilized by the optimization solver. 
    This highlights the efficiency of the proposed approach in balancing pre-computation and 
    real-time optimization.
    The value of the optimal cost function is $J^* = 28.71$, and the solver converged to the optimal solution
    with a final tolerance of $10^{-8}$.
    The optimization process produced the results shown in \cref{fig:example_1_combined}. 
    The vertical dashed lines in \cref{fig:example_1_combined} represent the switching instants, 
    determined by the optimal values of $\delta_i$. 
    \Cref{fig:example_1_paper_NAHS_optimal_state} illustrates the optimal state evolution, obtained by applying 
    the optimal control input $u^*$ and the optimal mode durations $\delta^*$ to the system dynamics. 
    The optimal control input $u^*$, which is piecewise constant as expected from the problem formulation, 
    is depicted in \cref{fig:example_1_paper_NAHS_optimal_input}.
    It is important to note that, as per \cref{assumption_1}, there are no abrupt changes in the state when transitioning 
    from one active mode to another. From \cref{fig:example_1_combined}, it is evident that the system's state successfully 
    converges to the origin. Additionally, the time partitions are non-uniform at the end of the optimization process, 
    demonstrating the algorithm's ability to adjust the switching instants to minimize the cost function. 
    This flexibility results in some modes having longer duration, while others have duration close to zero, 
    as depicted in \cref{fig:example_1_paper_NAHS_optimal_modes}.
    In \cref{fig:example_1_paper_NAHS_optimal_modes}, we highlight the active modes with duration exceeding 0.05 seconds. 
    This demonstrates the optimization algorithm's ability to refine the switching sequence, ensuring an efficient state transition 
    while effectively minimizing the cost function.
    
    \begin{figure}
        \centering
        \includegraphics[width=0.5\columnwidth, keepaspectratio]{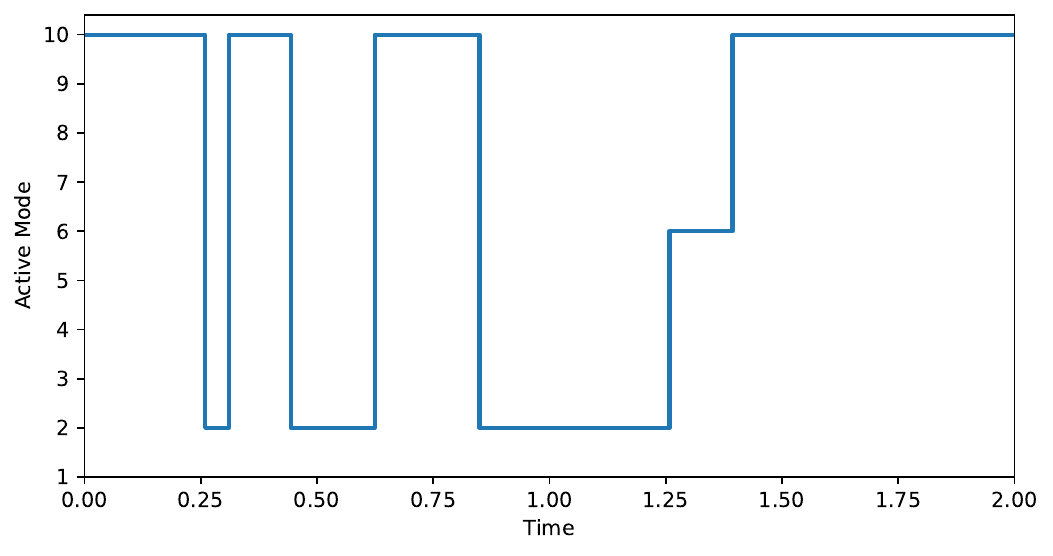}
        \caption{Optimized sequence of the active modes during the system's evolution.}
        \label{fig:example_1_paper_NAHS_optimal_modes}
    \end{figure}
    }
    \rev{\subsection{Example 2}
    In this example, we want to demonstrate the potential of our algorithm on a practical application.
    We consider a spring-mass system, which is a well-known example in control theory.
    We assume that the system can switch between two different configurations, a stance phase and a flight phase.
    The dynamics of the system are described by the following equations:
    \begin{equation*} \label{eq:example_2}
        \dot{x}(t) = {}^i A x(t) + {}^i B u(t) \quad i = \{1,\, 2\}.
    \end{equation*}
    The matrices are defined as follows:
    \[
    \begin{aligned} 
        {}^1 A &= \begin{bmatrix} 
            0 & 1 \\ 
            -\frac{k}{m} & -\frac{b}{m} \\
        \end{bmatrix}, & 
        {}^1 B &= \begin{bmatrix} 
            0 & 0 & 0 \\ 
            \frac{1}{m} & \frac{k}{m} & -1 \\
        \end{bmatrix}, & 
        {}^2 A &= \begin{bmatrix} 
            0 & 1 \\ 
            0 & 0 \\ 
        \end{bmatrix}, & 
        {}^2 B &= \begin{bmatrix} 
            0 & 0 & 0 \\ 
            0 & 0 & -1 \\ 
        \end{bmatrix},
    \end{aligned}
    \]
    with $m$ being the mass, $k$ being the spring constant, and $b$ being the damping coefficient.
    The state vector $x(t)$ is defined as $x(t) = [y(t), \dot{y}(t)]^\top$, 
    where $y(t)$ is the position of the mass, $\dot{y}(t)$ is its velocity.
    The control input $u(t)$ is defined as $u(t) = [F, l_0, g]^\top$, where $F$ is the force applied to the mass, $l_0$ is the rest length of the spring, and $g$ is the gravity acceleration.
    To fully define the system, we set the following parameters:
    \begin{equation*} \label{eq:parameters}
        m = 1.5\,\textnormal{kg}, \quad k = 100\,\textnormal{N/m},  \quad b = 25\,\textnormal{Ns/m}\quad l_0 = 0.5\,\textnormal{m}, \quad g = 9.81\,\textnormal{m/s}^2.
    \end{equation*}
    The objective is to regulate \eqref{eq:example_2}, transitioning the mass position from $y(0) = 0.4\,\textnormal{m}$ to $y(t_{\textnormal{f}}) = 0.7\,\textnormal{m}$ 
    within a time horizon of 1 second, i.e., $t_0 = 0\,\textnormal{s}$ and $t_\textnormal{f} = 1\,\textnormal{s}$. 
    The control input is constrained within the bounds $u_{\textnormal{min}} = \left[-100, l_0, g\right]$ and $u_{\textnormal{max}} = \left[100, l_0, g\right]$, 
    meaning that the control input can be used to apply a force of up to 100 N in either direction, while also allowing for the spring's rest length and gravity to be considered.
    \begin{figure}[ht]
        \centering
        \begin{subfigure}[t]{0.48\textwidth}
            \centering
            \includegraphics[width=\textwidth, keepaspectratio]{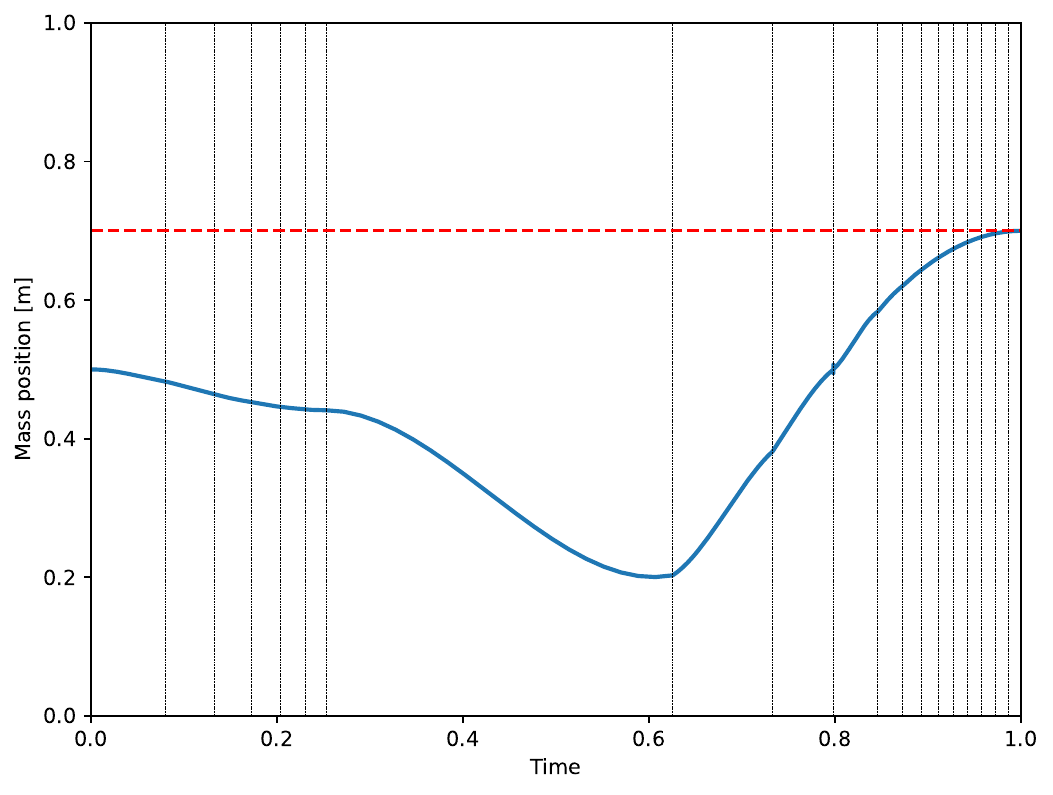}
            \caption{Optimal mass position evolution.}
            \label{fig:example_smd_paper_NAHS_spring_trajectory}
        \end{subfigure}
        \hfill
        \begin{subfigure}[t]{0.48\textwidth}
            \centering
            \includegraphics[width=\textwidth, keepaspectratio]{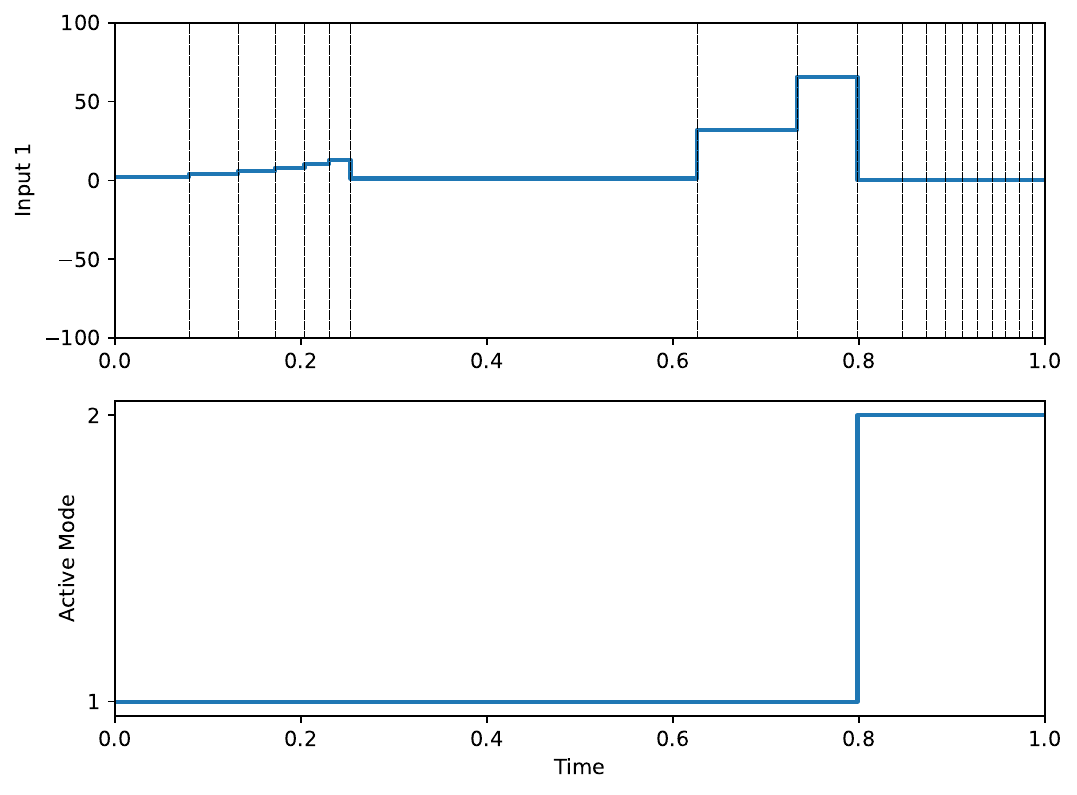}
            \caption{Optimal control input compared to the active modes in the interval $\left[t_0, t_{\textnormal{f}}\right]$.}
            \label{fig:example_smd_paper_NAHS_optimal_input}
        \end{subfigure}
        \caption{Results of the optimization: (a) Mass position evolution and (b) Control input and active modes evolution.}
        \label{fig:example_smd_paper_NAHS_combined}
    \end{figure}
    For this example, we define a switching sequence where each mode remains active for a fixed number of time intervals before transitioning 
    to the next mode. 
    This approach allows the solver to utilize multiple control inputs within the same active mode to help the system in reaching the desired state. 
    Specifically, we set $N_{\textnormal{I}}= 20$, where the first 10 modes
    correspond to the pair $\left({}^{1} A, {}^{1} B\right)$, representing the stance phase,
    and the remaining 10 modes correspond to the pair $\left({}^{2} A, {}^{2} B\right)$, representing
    the flight phase. To ensure sufficient flexibility in the optimization process, we
    choose $n_{\textnormal{c}}= 3$, allowing the solver to effectively regulating the system.
    Additionally, to maintain physical feasibility, we impose state constraints
    based on the system's phase: $y(t) \leq l_{0}$ during the stance phase to prevent
    overstretching the spring, and $y(t) \geq l_{0}$ during the flight phase to
    ensure the mass remains above the rest position. These constraints ensure realistic
    and physically meaningful optimization results.    We calculate \cref{eq:J}, \cref{eq:dJ_u}, and \cref{eq:dJ_delta} using $\tilde{Q} = 10 I_{n\times n}$, $R = 100 I_{m\times m}$, 
    and zero terminal cost.
    The desired position at the final time is enforced as a constraint, ensuring the system reaches the target state within the specified time horizon.
    
    The total solving time for the optimization problem is 17.73 seconds, with 14.71 seconds dedicated to offline computations and 3.02 seconds utilized 
    by the optimization solver.
    The increased solving time compared to the previous example is attributed to the additional constraints imposed on the problem, despite having fewer switching instants.
    The optimization process yielded a solution in which not all possible active modes were utilized, as evidenced by the results in \cref{fig:example_smd_paper_NAHS_combined}.
    
    \Cref{fig:example_smd_paper_NAHS_spring_trajectory} illustrates the spring deformation evolution, 
    obtained by applying the optimal control input $u^*$ and the optimal mode durations $\delta^*$ to the system dynamics.
    The vertical dashed lines in \cref{fig:example_smd_paper_NAHS_spring_trajectory} represent the switching instants,
    determined by the non-zero optimal values of $\delta_i$, while the red dashed line indicates the desired mass position, 
    set to $y(t_\textnormal{f}) = 0.7\,\textnormal{m}$.

    The optimal control input $u^*$, depicted in \cref{fig:example_smd_paper_NAHS_optimal_input}, is piecewise constant, 
    as expected from the problem formulation. Initially, the control input remains relatively low, 
    allowing the system to gradually approach the desired trajectory. However, during the final phase, 
    the control input increases significantly to exert the necessary force for achieving the desired position within the given time horizon. 
    This behavior aligns with the system's need to overcome the spring's resistance and reach the target position efficiently.
    From \cref{fig:example_smd_paper_NAHS_spring_trajectory}, it is evident that the spring stores energy during the stance phase 
    and releases it during the flight phase, enabling the system to reach the desired position. 
    This dynamic behavior is further highlighted in \cref{fig:example_smd_paper_NAHS_optimal_input}, 
    which shows the system's switching between the two modes. The transitions between the stance and flight phases 
    are optimized to ensure efficient energy utilization and precise trajectory tracking.
    For clarity, \cref{fig:example_smd_paper_NAHS_combined} illustrate multiple non-zero intervals. 
    However, there is only a single mode transition in reality. 
    This is because each mode requires specific control actions to satisfy the switching conditions necessary for the mass to achieve the desired position.
    Working with a physical system provides the advantage of analyzing results within a realistic context. 
    Notably, during short time intervals, represented by closely spaced black dashed lines, the control input has minimal impact on the system's state. 
    In such cases, the control input can be effectively neglected without significantly altering the system's behavior. 
    Moreover, this phenomenon does not influence the value of the cost function, ensuring the robustness of the optimization results. 
    }
    \subsection{Example 3}
    From \cref{corollary_1,corollary_2}, we state that our approach can work also with autonomous switched linear systems.
    To showcase this feature, we consider the autonomous linear switched system from \cite{stellato2017second, example} described by
    \begin{equation*}
        \dot{x}(t) = {}^i A x(t) \quad i = \{1,\, 2\}
    \end{equation*}
    
    with the following state matrices
    \begin{align*}
        {}^1 A = 
        \begin{bmatrix}
            -1 & 0 \\
            1 & 2
        \end{bmatrix}
        \quad
        {}^1 B = 
        \begin{bmatrix}
            0 \\
            0 
        \end{bmatrix},
        {}^2 A = 
        \begin{bmatrix}
            1 & 1 \\
            1 & -2
        \end{bmatrix}
        {}^2 B = 
        \begin{bmatrix}
            0 \\
            0 
        \end{bmatrix}
    \end{align*}
    
    Note that for this example, $N_\textnormal{I}$ equals 2, and both ${}^1 A$ and ${}^2 A$ describe unstable dynamics.
    We consider $t_0 = 0$, $t_\textnormal{f} = 1$s and we choose $n_\textnormal{c} = 3$.
    According to \cref{preliminaries}, the system's sequence of active modes is $\sigma = \{1,\,2,\,1,\,2,\,1,\,2\}$ with initial configuration $x_0 = [1,\,1]^\top$.
    We aim to find switching times that optimize the quadratic cost function  $\int_{t_0}^{t_\textnormal{f}}  \left(x^{\top}(t) \Tilde{Q} x(t)\right)\,  \textnormal{d}t$, where $\Tilde{Q}=I$ and the terminal cost is zero.
    The problem converged to an optimal solution quickly, producing the as a result $\tau^* = \left[0.100,\, 0.297,\, 0.433,\, 0.642,\, 0.767 \right]^\top$, which is the same obtained in \cite{stellato2017second, example}.
    
    This example demonstrates that even though our new approach considers the presence of a control input, it still works with autonomous systems. 

\section{Conclusions}
This paper presented a new framework that performs simultaneous optimization of switched linear systems over control input and switching instants.
The framework leverages a matrix exponential discretization over a partitioned time grid, parameterized through the duration of each partition. 
The paper showed that the cost function and gradient computation share many common terms that can be calculated offline before optimization, which reduces the computation cost and enhances the implementation efficiency.
Numerical integration is unnecessary for state propagation as the state evolution is already incorporated in the cost function.  
An example demonstrated the effectiveness of our approach with autonomous systems.
Another example showed that the algorithm finds optimal solutions for control inputs and switching instants in non-autonomous systems.
Future work may extend to nonlinear systems, including second-order derivatives.

\appendix

\section{Proof of Lemma \ref{lemma_1} and \ref{lemma_2}} 	\label{proof:lemma_1_2}
    \begin{proof}(Proof of Lemma \ref{lemma_1})
        Using \cref{eq:transition_matrix}, we can write $\Phi(\tau_{\ell}, \tau_a)$ as:
        \begin{equation}
            \label{eq_comp}
            \begin{split}
                &\Phi(\tau_{\ell}, \tau_a) =  
            \begin{bmatrix}
                \Phi_\textnormal{A}(\tau_{\ell}, \tau_a) & \Phi_\textnormal{F}(\tau_{\ell}, \tau_a) \\
                \underline{0} & 1
            \end{bmatrix}
            = \Phi(\tau_{\ell}, \tau_k) \Phi(\tau_k, \tau_a)  \\
            &=
            \begin{bmatrix}
                \Phi_\textnormal{A}(\tau_{\ell}, \tau_k) & \Phi_\textnormal{F}(\tau_{\ell}, \tau_k) \\
                \underline{0} & 1
            \end{bmatrix}
            \begin{bmatrix}
                \Phi_\textnormal{A}(\tau_k, \tau_a) & \Phi_\textnormal{F}(\tau_k, \tau_a) \\
                \underline{0} & 1
            \end{bmatrix} \\
            &= 
            \begin{bmatrix}
                \Phi_\textnormal{A}(\tau_{\ell}, \tau_k)\Phi_\textnormal{A}(\tau_k, \tau_a) & \Phi_\textnormal{A}(\tau_{\ell}, \tau_k)\Phi_\textnormal{F}(\tau_k, \tau_a)+ \Phi_\textnormal{F}(\tau_{\ell}, \tau_k)\\
                \underline{0} & 1
            \end{bmatrix},
            \end{split}
        \end{equation}
        where $\tau_k$ is a switching instant such that $\tau_a \leq \tau_k \leq \tau_l$.
        \rev{Recalling \cref{state}, \cref{eq_comp}, and two generic states $x_a$, $x_\ell$, one has
        \begin{equation} \label{comp}
            \begin{split}
                \overline{x}_{\ell} &= \Phi(\tau_{\ell}, \tau_a) \overline{x}_a = 
                \begin{bmatrix}
                    \Phi_\textnormal{A}(\tau_{\ell},\, \tau_k)\Phi_\textnormal{A}(\tau_{k},\, \tau_a) x_a + \Phi_\textnormal{A}(\tau_{\ell},\, \tau_k)\Phi_\textnormal{F}(\tau_{k},\, \tau_a) + \Phi_\textnormal{F}(\tau_{\ell},\, \tau_k) \\
                    1
                \end{bmatrix} ,
            \end{split}
        \end{equation}
        }
        where the first row represents the state evolution from $x_a$ to $x_\ell$.
        To complete the proof we need to rewrite explicitly $\Phi_{\textnormal{A}}$ and $\Phi_{\textnormal{F}}$ in \cref{comp}.
        Recalling \cref{phi_a}, \cref{phi_f}, and introducing an intermediate state $x_k$, we can express the state evolution from $x_a$ to $x_k$ and from $x_k$ to $x_{\ell}$ through the Lagrange formula \cref{eq:lagrange_formula}, i.e.,
        \begin{equation} \label{xi_xa}
            \begin{split}
                x_k &= \prod_{q=a}^{k-1} e^{A_q \delta_q} x_a + \sum_{q=a}^{k-1}\left(\prod_{s=q+1}^{k-1}  e^{A_s \delta_s}\right) \Phi_{\textnormal{F}_q}u_q = \Phi_\textnormal{A}(\tau_k,\, \tau_a) x_a + \Phi_\textnormal{F}(\tau_k,\, \tau_a)
            \end{split}
        \end{equation}
        
        \begin{equation} \label{xl_xi}
            \begin{split}
                x_{\ell} &= \prod_{q=i}^{\ell-1} e^{A_q \delta_q} x_k + \sum_{q=i}^{\ell-1}\left(\prod_{s=q+1}^{\ell-1} e^{A_s \delta_s}\right) \Phi_{\textnormal{F}_q}u_q = \Phi_\textnormal{A}(\tau_{\ell},\, \tau_k) x_k + \Phi_\textnormal{F}(\tau_{\ell},\, \tau_k).
            \end{split}
        \end{equation}
        
        By inserting \cref{xi_xa} into \cref{xl_xi} and performing the state augmentation, we obtain \cref{comp} which completes the proof.
    \end{proof}
    
    \begin{proof}(Proof of Lemma \ref{lemma_2})
        Differentiating the transition matrix \cref{eq:transition_matrix} w.r.t. $\delta_k \in \R$, we get
        \begin{equation} \label{chain_delta}
            \begin{split}
                \frac{\partial \Phi}{\partial \delta_k}(\tau_{k+1}, \tau_a) &= 
                \begin{bmatrix}
                    \frac{\partial \Phi_\textnormal{A}}{\partial \delta_{k}}(\tau_{k+1}, \tau_a) & \frac{\partial \Phi_\textnormal{F}}{\partial \delta_{k}}(\tau_{k+1}, \tau_a) \\
                    \underline{0} & 0
                \end{bmatrix}
                = M_k \Phi(\tau_{k+1}, \tau_a).
            \end{split}
        \end{equation}
    
        We can rewrite \cref{phi_a} using \cref{composition_rule} and \cref{eq:lagrange_formula}:
        \begin{equation} \label{phi_a_dec}
            \Phi_\textnormal{A}(\tau_{k+1},\, \tau_a) = e^{A_k\delta_k} \Phi_\textnormal{A}(\tau_k, \tau_a).
        \end{equation}
        
        Recalling and differentiating \cref{phi_a_dec} w.r.t. $\delta_k$, one has
        \begin{equation} \label{d_phi_a}
            \begin{split}
                \frac{\partial\Phi_\textnormal{A}(\tau_{k+1},\, \tau_a)}{\partial \delta_k} &= A_k e^{A_k\delta_k}\Phi_\textnormal{A}(\tau_k,\,\tau_a)
                = A_k \Phi_\textnormal{A}(\tau_{k+1},\,\tau_a);
            \end{split}
        \end{equation}
        where the property of the matrix exponential $\frac{\partial e^{Yc}}{\partial c} = Ye^{Yc}$ is used, where $Y \in \R^{n \times n}$ and $c \in R$.
        The forced evolution term \cref{phi_f} needs a few more considerations related to the autonomous evolution term.
        To understand how $\delta_k$ affects the forced evolution, we rewrite \cref{phi_f} as follows 
        \begin{equation} \label{new_phi_f}
            \Phi_\textnormal{F}(\tau_{k+1}, \tau_a) = \sum_{q=a}^{k}\left(\prod_{s=q+1}^{k} e^{A_s \delta_s}\right) \Phi_{\textnormal{F}_q} u_q = \sum_{q=a}^{k} P(q) \Phi_{\textnormal{F}_q}u_q,
        \end{equation}
        
        where $P(q) \in \R^{n \times n}$ represents the autonomous evolution of the previous modes and $\Phi_{\textnormal{F}_q}$ is the integral term.
        From \cref{d_phi_a}, we get $\frac{\partial P(q)}{\partial \delta_k} = A_k P(q)$ if $q \neq k$ and $\frac{\partial P(q)}{\partial \delta_k} =0$ if $q = k$ because $P(k) = I$.
        In addition, $\frac{\partial \Phi_{\textnormal{F}_q} u_q}{\partial \delta_k} =0$ if $q \neq k$, because it does not depend on $\delta_k$.
        Finally, we only need to compute $\frac{\partial \Phi_{\textnormal{F}_k} u_k}{\partial \delta_k}$, thus
        \begin{align} 
            \frac{\partial \Phi_{\textnormal{F}_k}}{\partial \delta_k} &= \frac{\partial}{\partial \delta_k}\left[ \int_{\tau_k}^{\delta_k+\tau_k}e^{A_k(\delta_k+\tau_k-t)}B_k dt \right] u_k \label{d_phi_f_i_eq_1}\\
            &= \int_{\tau_k}^{\delta_k+\tau_k}A_ke^{A_k(\delta_k+\tau_k-t)}B_k dt \; u_k + B_k u_k  \label{d_phi_f_i_eq_2} \\
            &= A_k \Phi_{\textnormal{F}_k}u_k + B_k u_k. \label{d_phi_f_i_eq_3}
        \end{align}
        
        In \cref{d_phi_f_i_eq_1}, we point out the dependency from $\delta_k$ by defining $\tau_{k+1} = \delta_k+\tau_k$. 
        In \cref{d_phi_f_i_eq_2}, we use the Leibniz integral rule to compute the derivative, and using \cref{eq:phi_f_i}, we rewrite the derivative in \cref{d_phi_f_i_eq_3}.
        Recalling \cref{new_phi_f}, we expand the derivative of the product $P(q) \Phi_{\textnormal{F}_q}u_q$.
        Then, recalling \cref{d_phi_a} and \cref{d_phi_f_i_eq_3}, we write \cref{d_phi_f_2}.
        \begin{align} 
            \frac{\partial \Phi_{\textnormal{F}_k}(\tau_{k+1}, \tau_a)}{\partial \delta_k} &=
            \sum_{q=a}^{k-1} \frac{\partial P(q)}{\partial \delta_k} \Phi_{\textnormal{F}_q}u_q + P(k) \frac{\partial \Phi_{\textnormal{F}_k}}{\partial \delta_k} u_k \label{d_phi_f} \\ 
            &= \sum_{q=a}^{k-1} A_k P(q)\Phi_{\textnormal{F}_q} u_q + A_k \Phi_{\textnormal{F}_k} u_k + B_k u_k \label{d_phi_f_1}\\
            &= A_k \sum_{q=a}^{k} P(q)\Phi_{\textnormal{F}_q} u_q + B_k u_k \label{d_phi_f_2}\\
            &= A_k \Phi_{\textnormal{F}_k}(\tau_{k+1}, \tau_a) + B_k u_k. \label{d_phi_f_3}
        \end{align}
        
        Rearranging the terms in \cref{d_phi_f_3}, we complete the first part of the proof.
        
        Considering the input as piecewise constant \cref{eq:piecewise_input}, the derivative of \cref{eq:transition_matrix} w.r.t. $u_{kj}$, that is the \textit{j}-th component of the control input of the \textit{k}-th mode, results in
        \begin{equation}
            \begin{split}
                \frac{\partial \Phi}{\partial u_{kj}}(\tau_{k+1}, \tau_a) &= 
                \begin{bmatrix}
                    \frac{\partial \Phi_\textnormal{A}}{\partial u_{kj}}(\tau_{k+1}, \tau_a) & \frac{\partial \Phi_\textnormal{F}}{\partial u_{kj}}(\tau_{k+1}, \tau_a) \\
                    \underline{0} & 0
                \end{bmatrix} 
                = 
                \begin{bmatrix}
                    \underline{0} & \Phi_{F_{kj}} \\
                    \underline{0} & 0
                \end{bmatrix} = H_{kj} = H_{kj} \Phi(\tau_{k+1}, \tau_a),
            \end{split}
            \label{sto}
        \end{equation}
        where $H_{kj}$ is defined in \cref{H_matrix}.
        $\frac{\partial \Phi_\textnormal{A}}{\partial u_{kj}} (\tau_{k+1}, \tau_a) = 0_{n \times n}$ because it does not depend on the control input.
        From \cref{new_phi_f}, we can decompose $\Phi_{\textnormal{F}_q}u_q$ as follows
        \begin{equation} \label{phi_f_decomp}
            \Phi_{\textnormal{F}_q}u_q = \sum_{j=1}^{m} \Phi_{\textnormal{F}_{qj}}u_{qj}.
        \end{equation}
        
        Rewriting \cref{new_phi_f} using \cref{phi_f_decomp} and differentiating it, leads us to
        \begin{equation}
            \begin{split}
                \frac{\partial \Phi_\textnormal{F}}{\partial u_{kj}}(\tau_{k+1}, \tau_a) &= \frac{\partial}{\partial u_{kj}}\sum_{q=a}^{k} P(q) \left( \sum_{j=1}^{m} \Phi_{F_{qj}}u_{qj} \right)
                = P(k) \Phi_{\textnormal{F}_{kj}} = \Phi_{\textnormal{F}_{kj}}.
            \end{split}
        \end{equation}
        
        This completes the proof.
    \end{proof}

\section{Proof of Theorem \ref{th_J}} \label{proof:th_1}
    In order to prove Theorem \ref{th_J}, we need the following lemma.
    \begin{lemma} \label{lemma_3}
        Given two switching instants $\tau_a$ and $\tau_{k+1}$ with $\tau_a \leq \tau_{k+1}$, the derivative of $S_a$ w.r.t. the switching interval, i.e., $\delta_k$, can be written as:
        \begin{equation} \label{Sa_d}
            \frac{\partial S_a}{\partial \delta_k} = \Phi(\tau_{k+1}, \tau_a)^{\top} C_k \Phi(\tau_{k+1}, \tau_a)\,,
        \end{equation}
        
        and the derivative of $S_a$ w.r.t. the control input, i.e., $u_{kj}$, can be written as:
        \begin{equation} \label{Sa_u}
            \frac{\partial S_a}{\partial u_{kj}} = \Phi(\tau_{k}, \tau_a)^{\top} D_{kj} \Phi(\tau_{k}, \tau_a) + \Phi(\tau_{k+1}, \tau_a)^{\top} N_{kj} \Phi(\tau_{k+1}, \tau_a)\,.
        \end{equation}
    \end{lemma}
    \begin{proof}
        To prove \cref{Sa_d}, we first split $S_a$ into:
        \begin{equation} \label{dSa}
            \frac{\partial S_a}{\partial \delta_k} = \frac{\partial P_a}{\partial \delta_k} + \frac{\partial F_a}{\partial \delta_k}.
        \end{equation}
        
        Expanding the first term of \cref{dSa}, we get
        \begin{equation}
            \begin{split} \label{dPa}
                \frac{\partial P_a}{\partial \delta_k} &= \frac{\partial}{\partial \delta_k} \left(\int_{\tau_k}^{\tau_{k+1}} \Phi(t,\, \tau_a)^\top Q \Phi(t,\, \tau_a) \textnormal{d}t\right)
                + \frac{\partial}{\partial \delta_k} \left(\int_{\tau_{k+1}}^{t_\textnormal{f}} \Phi(t,\, \tau_a)^\top Q \Phi(t,\, \tau_a) \textnormal{d}t\right)
                = D_{\delta_k}P_a' + D_{\delta_k}P_a''
            \end{split}
        \end{equation}
        
        We notice that the integral between $\tau_a$ and $\tau_k$ does not depend on $\delta_k$ for $t \in \left[\tau_a,\, \tau_k \right]$.
        Considering the left-end term of \cref{dPa}, we can rewrite it as
        \begin{align}
            & D_{\delta_k}P_a'
            = \Phi(\tau_k,\, \tau_a)^\top \frac{\partial}{\partial \delta_k} \biggl(\int_{0}^{\delta_k} \Phi(\eta+\tau_k,\, \tau_k)^\top Q 
            \Phi(\eta+\tau_k,\, \tau_k) \textnormal{d}\eta \biggr) \Phi(\tau_k,\, \tau_a) \nonumber \\
            &= \Phi(\tau_k,\, \tau_a)^\top \Phi(\tau_{k+1},\, \tau_k)^\top Q \Phi(\tau_{k+1},\, \tau_k) \Phi(\tau_k,\, \tau_a)
            = \Phi(\tau_{k+1},\, \tau_a)^\top Q \Phi(\tau_{k+1},\, \tau_a), \label{dPa_1}
        \end{align}
        using Lemma \ref{lemma_1}, we highlight that $\Phi(\tau_k,\, \tau_a)$ does not depend on $\delta_k$ and its derivative is zero.
        Moreover, we use the change of variables $\eta = t-	\tau_k$ to point out the dependency on $\delta_k$.
        Next, we rewrite the outermost term of \cref{dPa} as follows
        \begin{equation}
            \begin{split}
                \label{dPa_2_1}
                &D_{\delta_k}P_a''
                = \frac{\partial}{\partial \delta_k} \biggl(\int_{\tau_{k+1}}^{t_\textnormal{f}} \Phi(\tau_{k+1},\, \tau_a)^\top \Phi(t,\, \tau_{k+1})^\top Q
                \Phi(t,\, \tau_{k+1}) \Phi(\tau_{k+1},\, \tau_a) \textnormal{d}t \biggr)\,.
            \end{split}
        \end{equation}
        Starting from \cref{dPa_2_1}, using Lemma \ref{lemma_1} we can take $\Phi(\tau_{k+1},\, \tau_a)$ outside of the integral and use the definition of $P_{k+1}$ in \cref{P_matrix} to obtain \cref{dPa_2_2}.
        \begin{align}
            &\frac{\partial}{\partial \delta_k} \left(\ \Phi(\tau_{k+1},\, \tau_a)^\top P_{k+1} \Phi(\tau_{k+1},\, \tau_a) \right) \label{dPa_2_2} \\
            &= \Phi(\tau_{k+1},\, \tau_a)^\top M_k^\top P_{k+1} \Phi(\tau_{k+1},\, \tau_a) \label{dPa_2_3} \\
            &\quad \quad \quad + \Phi(\tau_{k+1},\, \tau_a)^\top P_{k+1} M_k \Phi(\tau_{k+1},\, \tau_a) \nonumber \\
            &= \Phi(\tau_{k+1},\, \tau_a)^\top \left(M_k^\top P_{k+1} + P_{k+1} M_k \right) \Phi(\tau_{k+1},\, \tau_a). \label{dPa_2_4}
        \end{align}
        
        It is worth noting from \cref{dPa_2_3} to \cref{dPa_2_4} we use Lemma \ref{lemma_2}.
        Using \cref{dPa_1} and \cref{dPa_2_4}, we rewrite \cref{dPa} as
        \begin{equation} \label{dPa_new}
            \frac{\partial P_a}{\partial \delta_k} = \Phi(\tau_{k+1},\, \tau_a)^\top \left(Q + M_k^\top P_{k+1} + P_{k+1} M_k \right) \Phi(\tau_{k+1},\, \tau_a).
        \end{equation}
        
        Recalling \cref{F_matrix}, $\frac{\partial F_a}{\partial \delta_k}$ can be expressed as
        \begin{align}
            \frac{\partial F_a}{\partial \delta_k} &= \frac{\partial}{\partial \delta_k} \biggl(\Phi(\tau_{k+1},\, \tau_a)^\top \Phi(t_\textnormal{f},\, \tau_{k+1})^\top E 
            \Phi(t_\textnormal{f},\, \tau_{k+1}) \Phi(\tau_{k+1},\, \tau_a) \biggr)  \label{dFa_1}  \\
            &= \frac{\partial}{\partial \delta_k} \biggl(\Phi(\tau_{k+1},\, \tau_a)^\top F_{k+1} \Phi(\tau_{k+1},\, \tau_a) \biggr) \label{dFa_3} \\ &
            = \Phi(\tau_{k+1},\, \tau_a)^\top \left[M_k^\top F_{k+1} + F_{k+1} M_k \right] \Phi(\tau_{k+1},\, \tau_a), \label{dFa_2}
        \end{align}
        where in \cref{dFa_1} we highlight the state evolution from $\tau_a$ to $\tau_{k+1}$, which is the only term that depends on $\delta_k$.
        At the end, by adding \cref{dPa_new}, \cref{dFa_2}, and exploiting \cref{definition_1}, we obtain
        \begin{equation}
            \frac{\partial S_a}{\partial \delta_k} = \Phi(\tau_{k+1},\, \tau_a)^\top \left(Q + M_k^\top S_{k+1} + S_{k+1} M_k \right) \Phi(\tau_{k+1},\, \tau_a).
        \end{equation}
        
        Recalling \cref{definition_2}, we get \cref{Sa_d}.
        
        Similarly, we can obtain the proof of \cref{Sa_u} by splitting $S_a$ into two terms
        \begin{equation} \label{dSa_u}
            \frac{\partial S_a}{\partial u_{kj}} = \frac{\partial P_a}{\partial u_{kj}} + \frac{\partial F_a}{\partial u_{kj}}.
        \end{equation}
        
        Likewise in \cref{dPa}, we split the equation into two integrals
        \begin{equation}
            \begin{split} \label{dPa_u}
                \frac{\partial P_a}{\partial u_{kj}} &= \frac{\partial}{\partial u_{kj}} \left(\int_{\tau_k}^{\tau_{k+1}} \Phi(t,\, \tau_a)^\top Q \Phi(t,\, \tau_a) \textnormal{d}t\right)
                + \frac{\partial}{\partial u_{kj}} \left(\int_{\tau_{k+1}}^{t_\textnormal{f}} \Phi(t,\, \tau_a)^\top Q \Phi(t,\, \tau_a) \textnormal{d}t\right) \\
                &= D_{u_{kj}}P_a' + D_{u_{kj}}P_a''
            \end{split}
        \end{equation}
        
        The state evolution from $\tau_a$ to $\tau_k$ does not depend on $u_{kj}$ and its derivative is zero.
        The first term of \cref{dPa_u} can be expressed as follows
        \begin{align}
            &D_{u_{kj}}P_a'
            = \Phi(\tau_k,\, \tau_a)^\top \int_{\tau_k}^{\tau_{k+1}} \frac{\partial}{\partial u_{kj}} \biggl(\Phi(t,\, \tau_k)^\top Q 
            \Phi(t,\, \tau_k) \biggr) \textnormal{d}t \; \Phi(\tau_k,\, \tau_a) \label{dPa_u1_new}
        \end{align}
        Recalling \cref{dPa_u1_new} and Definition \ref{definition_2}, performing mathematical computation lead to
        \begin{equation}
            \label{dPa_u_1}
            D_{u_{kj}}P_a' = \Phi(\tau_{k},\, \tau_a)^\top D_{kj} \Phi(\tau_{k},\, \tau_a). 
        \end{equation}
        
        We can reformulate \cref{dPa_u2_new} using Lemma \ref{lemma_1}.
        Then, we differentiate w.r.t. $u_{kj}$ applying Lemma \ref{lemma_2} and we obtain \cref{dPa_u_2}.
        \begin{align}
            D_{u_{kj}}P_a''
            &= \frac{\partial}{\partial u_{kj}} \biggl(\Phi(\tau_{k+1},\, \tau_a)^\top \int_{\tau_{k+1}}^{t_\textnormal{f}}  \Phi(t,\, \tau_{k+1})^\top Q 
            \Phi(t,\, \tau_{k+1}) \textnormal{d}t \; \Phi(\tau_{k+1},\, \tau_a) \biggr) \label{dPa_u2_new}\\
            &= \Phi(\tau_{k+1},\, \tau_a)^\top \left(H_{kj}^\top P_{k+1} + P_{k+1} H_{kj} \right) \Phi(\tau_{k+1},\, \tau_a) \label{dPa_u_2}
        \end{align}
        
        Recalling \cref{dPa_u_1} and \cref{dPa_u_2}, we can rewrite \cref{dPa_u} as 
        \begin{equation} \label{dPa_new_u}
            \begin{split}
                &\frac{\partial P_a}{\partial u_{kj}} = \Phi(\tau_{k},\, \tau_a)^\top D_{kj} \Phi(\tau_{k},\, \tau_a) + \Phi(\tau_{k+1},\, \tau_a)^\top
                \left(H_{kj}^\top P_{k+1} + P_{k+1} H_{kj} \right) \Phi(\tau_{k+1},\, \tau_a).
            \end{split}
        \end{equation}
        
        Thus, recalling \cref{dSa_u}, we can rewrite $\frac{\partial F_a}{\partial u_{kj}}$ as in \cref{dFa_u_1}, exploiting Lemma \ref{lemma_1} and the definition of $F_{k+1}$ from \cref{F_matrix}.
        Hence, we obtain \cref{dFa_u_2} by performing the differentiation and exploiting Lemma \ref{lemma_2}.
        \begin{align}
            &\frac{\partial F_a}{\partial u_{kj}} = \frac{\partial}{\partial u_{kj}} \biggl(\Phi(\tau_{k+1},\, \tau_a)^\top \Phi(t_\textnormal{f},\, \tau_{k+1})^\top E
            \Phi(t_\textnormal{f},\, \tau_{k+1}) \Phi(\tau_{k+1},\, \tau_a) \biggr) \label{dFa_u_1} \\
            &= \frac{\partial}{\partial u_{kj}} \biggl(\Phi(\tau_{k+1},\, \tau_a)^\top F_{k+1} \Phi(\tau_{k+1},\, \tau_a) \biggr) \label{dFa_u_2}  \\
            &= \Phi(\tau_{k+1},\, \tau_a)^\top \left(H_{kj}^\top F_{k+1} + F_{k+1} H_{kj} \right) \Phi(\tau_{k+1},\, \tau_a), \label{dFa_u}
        \end{align}
        
        Finally, adding \cref{dPa_new_u} and \cref{dFa_u}, one has
        \begin{equation} \label{dSa_du_proof}
            \begin{split}
                &\frac{\partial S_a}{\partial u_{kj}} = \Phi(\tau_{k},\, \tau_a)^\top D_{kj} \Phi(\tau_{k},\, \tau_a)
                + \Phi(\tau_{k+1},\, \tau_a)^\top \left(H_{kj}^\top P_{k+1} + P_{k+1} H_{kj} \right) \Phi(\tau_{k+1},\, \tau_a) \\
                &+ \Phi(\tau_{k+1},\, \tau_a)^\top \left(H_{kj}^\top F_{k+1} + F_{k+1} H_{kj} \right) \Phi(\tau_{k+1},\, \tau_a). 
            \end{split}
        \end{equation}		
        
        Using Definition \ref{definition_1} and \ref{definition_2} we complete the proof.	  
    \end{proof}
    
    \begin{proof}(Theorem \ref{th_J})
        From \cref{eq:J}, we substitute the definition of $G_0$, i.e., \cref{G}, and $S_0$, i.e., \cref{S_matrix}, into \cref{eq:obj_finite_dim_prob_2}.
        At this point, the proof is completed.
    \end{proof}
    
\section{Proof of Theorem \ref{th_dJ}} \label{proof:th_2}
    Recalling Lemma \ref{lemma_3} and \cref{state}, we can rewrite \cref{eq:dJ_u} as
    \begin{align}
        \frac{\partial J(u, \delta)}{\partial u_{kj}} &= \frac{\partial G_0}{\partial u_{kj}} + \overline{x}_0^{\top} \frac{\partial S_0}{\partial u_{kj}} \overline{x}_0  \nonumber\\
        &= u_{kj} R_{jj} \delta_k + \overline{x}_0^{\top} \biggl( \Phi(\tau_{k},\,\tau_0)^{\top} D_{kj} \Phi(\tau_{k},\,\tau_0) 
        + \Phi(\tau_{k+1},\,\tau_0)^{\top} N_{kj} \Phi(\tau_{k+1},\,\tau_0) \biggr) \overline{x}_0 \nonumber \\
        &= u_{kj} R_{jj} \delta_k + \overline{x}_{k}^{\top} D_k \overline{x}_{k} + \overline{x}_{k+1}^{\top} N_\textnormal{k} \overline{x}_{k+1}. \label{proof_dJ_u}
    \end{align}
    
    Next, recalling \cref{state}, \cref{second_term}, and Lemma \ref{lemma_3}, we aim to prove \cref{eq:dJ_delta}, i.e.,
    \begin{align}
        \frac{\partial J(u, \delta)}{\partial \delta_k} &= \frac{\partial G_0}{\partial \delta_k} + \overline{x}_0^{\top} \frac{\partial S_0}{\partial \delta_k} \overline{x}_0 \nonumber\\
        &= u_k^\top R u_k + \overline{x}_0^{\top} \Phi(\tau_{k+1},\,\tau_0)^{\top} C_k \Phi(\tau_{k+1},\,\tau_0) \overline{x}_0 \nonumber\\
        &= u_k^\top R u_k + \overline{x}_{k+1}^{\top} C_k \overline{x}_{k+1}. \label{proof_dJ_d}
    \end{align}
    
    This completes the proof. \hfill $\blacksquare$
    
\section{Proof of Preposition \ref{prep}} \label{proof_prep}
Recalling Definition \ref{definition_2} and setting $i = N+1$, we obtain \cref{SN}.
The recursive equation in \cref{Si} can be derived using Definition \ref{definition_2} as follows:
\begin{align}
    S_k &= \int_{\tau_k}^{\tau_{k+1}} \Phi(t,\, \tau_k)^\top Q \Phi(t,\, \tau_k) \textnormal{d}t
    + \Phi(\tau_{k+1},\, \tau_k)^\top \left[\int_{\tau_{k+1}}^{t_\textnormal{f}} \left( \Phi(t,\, \tau_k)^\top Q \Phi(t,\, \tau_k)  \right) \textnormal{d}t
    + \Phi(t_{\textnormal{f}}, \tau_{k+1})^\top E \Phi(t_{\textnormal{f}}, \tau_{k+1})\right] \Phi(\tau_{k+1},\, \tau_k) \label{proof_Si_1} \\
    &= \int_{0}^{\delta_k} \Phi(\eta+\tau_k,\, \tau_k)^\top Q \Phi(\eta+\tau_k,\, \tau_k) \textnormal{d}\eta + \Phi(\tau_{k+1},\, \tau_k)^\top S_{k+1} \Phi(\tau_{k+1},\, \tau_k) \label{proof_Si_2}
\end{align}

In \cref{proof_Si_1}, we can recognize $S_{k+1}$ by highlighting the matrices $\Phi(\tau_{k+1},\, \tau_k)^\top$ and $\Phi(\tau_{k+1},\, \tau_k)$ and by putting them outside the integral because they do not depend on $t$.
Then, in \cref{proof_Si_2}, we apply the change of variables $\eta = t - \tau_k$ and recalling Definition \ref{definition_2} the proof is completed. \hfill $\blacksquare$



\bibliographystyle{elsarticle-num}
\bibliography{refs}


\end{document}